\documentclass{article}
\usepackage{graphicx}
\usepackage[letterpaper,top=2cm,bottom=2cm,left=3cm,right=3cm,
marginparwidth=1.75cm]{geometry}
\usepackage{mathtools}
\usepackage{bm, bbm}
\usepackage{amsmath, amsthm}
\usepackage{amssymb}
\usepackage{dsfont}
\usepackage{booktabs}
\usepackage{xcolor}
\usepackage[table]{xcolor} 
\usepackage{algorithm}
\usepackage{algpseudocode}
\usepackage{enumitem}
\usepackage{multirow}
\usepackage{float}
\usepackage[colorlinks,citecolor=blue,linkcolor=blue,urlcolor=blue]{hyperref}
\usepackage{tikz}
\usetikzlibrary{arrows.meta,positioning}
\newcommand{\gap}{\vspace{0.05in}}
\newcommand{\epc}{\hspace{1pc}}
\newcommand{\wh}{\widehat}

\newcommand{\thalf}{\frac{1}{2}}
\newcommand{\onebld}{{\bf 1}}
\usepackage{tikz}

\newtheorem{theorem}{Theorem}
\newtheorem{proposition}{Proposition}

\newtheorem{remark}{Remark}
\newtheorem{assumption}{Assumption}
\newtheorem{example}{Example}

\title{A New Perspective on Clustering: A Mixed-norm Model \\ 
and its Solution by Progressive Integer Programming}

\author{{\bf Junyi Liu}\footnote{Department of Industrial 
Engineering, Tsinghua University, Beijing, 100084, China.
{\tt Email: junyiliu@tsinghua.edu.cn}}, 
{\bf Yulin Peng}\textsuperscript{*}\footnote{\textsuperscript{*}Corresonding author. The Daniel J.\ Epstein Department 
of Industrial and Systems Engineering, University of Southern 
California, Los Angeles,
California 90089, U.S.A. {\tt Email: yulinpen@usc.edu}}, 
{\bf Yao Xie}\footnote{H.\ Milton Stewart School of Industrial 
and Systems Engineering, Georgia Institute of Technology, Atlanta,
Georgia 30332-0205, U.S.A.  {\tt Email: yaoxie@isye.gatech.edu}}, 
{\bf Yancheng Yuan}\textsuperscript{*}\footnote{\textsuperscript{*}Corresonding author. Department of Applied 
Mathematics, Hong Kong Polytechnic University, Kowloon, Hong Kong.
{\tt Email: yancheng.yuan@polyu.edu.hk}}, 
and {\bf Jong-Shi Pang}\footnote{Postal address same as Yulin Peng.
{\tt Email: jongship@usc.edu}}}

\date{\today}

\begin{document}

\maketitle

\vspace{-0.2in}

\begin{abstract}
\noindent Extending the classical $K$-means and $K$-medians 
models for clustering, this paper introduces a
$\ell_{p,q}$ mixed-norm clustering model where the centroid 
updates and the cluster assignments are under the $\ell_p$ norm 
and $\ell_q$ norm, respectively. The model 
is formulated as a mixed-integer program (MIP) with constraints stated 
in terms of Heaviside composite functions that describe the 
nearest-center assignments. % step.
The framework recovers $K$-means and
$K$-medians when $p=q=2$ and $p=q=1$, respectively, and yields new
models when $p\ne q$. To address the computational challenges, we
develop a progressive integer programming (PIP) method that adaptively
fixes confident assignments and solves a sequence of restricted 
mixed-integer subproblems.  For $q=1$, we develop a convex inner 
approximation of the difference-of-convex constraints in the 
restricted subproblems, for which a global solution can
be computed. Importantly, we establish the connection between 
the local minimizer 
and the strong center-local minimizer of the mixed-norm clustering problem and the (global) optimal solution of the restricted subproblems under certain assumptions. This connection provides a practical certificate of the local minimizer of the nonconvex mixed-norm clustering model. We further develop effective techniques for constructing adaptive fixing sets and working sets for the case of $q = 1$.  Extensive numerical experiments demonstrate the superior performance of the mixed-norm clustering model and the efficiency of the PIP method for solving the MIP model, which may be intractable otherwise. In particular, the $\ell_{2,1}$ mixed-norm clustering 
model is effective under
coordinate-sparse, mean-balanced contamination, whereas the $\ell_{1,2}$ 
mixed-norm clustering model is
preferred under dense coordinatewise Cauchy contamination. 
The numerical results also show
that PIP can escape poor alternating solutions and obtain substantially 
better feasible clustering, while preserving strong
warm starts when no improvement is found.  These results position 
mixed-norm clustering as a flexible modeling framework and PIP as 
an efficient algorithm for solving it.
\end{abstract}

\vspace{-0.2in}

\section{Introduction}\label{sec:intro}

Clustering is a fundamental task in unsupervised learning: 
given observations
$\{\xi^s\}_{s=1}^N\subset\mathbb{R}^d$ and a prescribed number 
of clusters $K>1$, the
goal is to partition the observations into groups whose members 
are similar to
one another and dissimilar to members of other groups.  
In centroid-based clustering, each group is represented by a 
center $c^k\in\mathbb{R}^d$, and a typical clustering algorithm 
jointly finds the cluster centroids and the cluster assignments 
of the data points by minimizing the sum of distances for each 
data point to its assigned centroid. The classical K-means 
and K-medians \cite{macqueen1967,Lloyd1982} are two prominent 
models which are widely applied in computer vision, healthcare, finance, renewable energy management, etc. They are known to be 
NP-hard \cite{mahajan2009planar, cohen2021inapproximability}, 
due to the combinatorial nature of the clustering problem. 
In practice, alternating algorithms with initialization 
enhancements \cite{arthur2007kmeans++} are 
popular, which iteratively refine $K$ representative centers 
$c^1, \cdots, c^K$ each in $\mathbb{R}^d$ and perform 
a corresponding assignment of points to centers.  We provide
more details of such 
an alternating method in Subsection~\ref{subsec:matched norm}. 

Two distinct norm-based modeling decisions are embedded in any such 
center-based method:
one such norm governs cluster assignments (how each point is matched 
to its nearest center)
and another norm governs the similarity objective (how within-cluster 
spread is measured). The two popular K-means and K-medians models use 
the same norm for both purposes. This coupling is natural, but it is 
not logically necessary.  The geometry that
best expresses cluster membership need not be the loss that best 
measures the quality of a fitted center.  For example, one may wish 
to retain Euclidean Voronoi assignments while using an 
absolute-deviation loss to reduce the
influence of outliers, or to form assignments in Manhattan
geometry.  Decoupling these two
roles therefore enlarges the modeling space while preserving the 
familiar center-and-assignment interpretation. An example shown in 
Figure~\ref{fig:claim1prime} demonstrates that proper decoupling 
of the two norms can achieve better clustering performance in practice.

\begin{figure}[t]
  \centering
  \includegraphics[width=\textwidth]{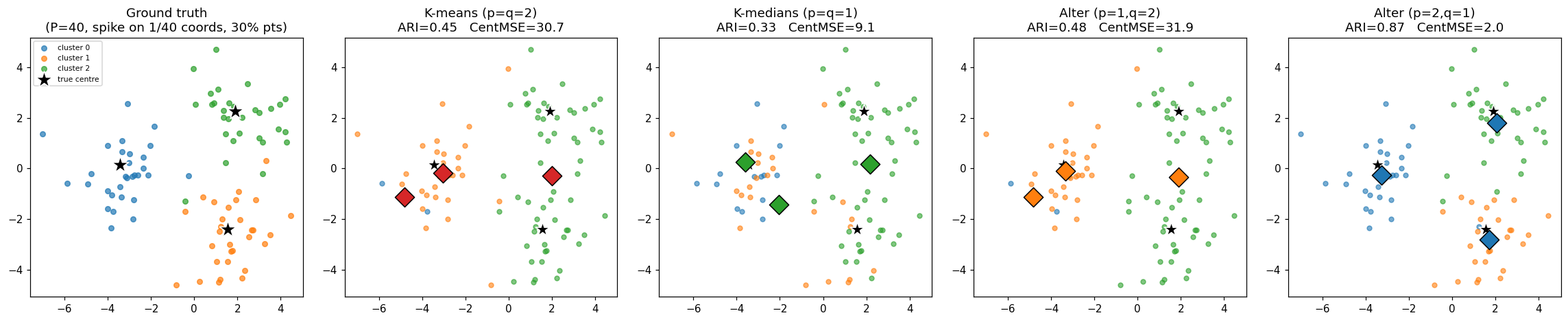}
  \caption{  A single representative instance of the sparse-corruption model
    ($K=3$, $P=40$, projected onto the plane of the three true centers).
    Stars: true centers; colored dots: each method's cluster assignment;
    diamonds: estimated centroids.  Only the mixed pair $(p{=}2,q{=}1)$
    recovers all three centers (ARI $0.87$, CentMSE $2.0$); the two $\ell_2$-%
    assignment methods strand a centroid in the empty middle, and K-medians
    keep closer centroids but still misassign a whole cluster.}
  \label{fig:claim1prime}
\end{figure}

Motivated by this observation, we
introduce the
$\ell_{p,q}$ mixed-norm clustering problem for 
$p, q \in \{ 1,2 \}$, where the within-cluster loss is based
on the $\ell_p$-norm and the cluster assignments
are based on the $\ell_q$-norm, with possibly $p \neq q$. 
Let $[ \, m \, ] \triangleq\{1,\ldots,m\}$ for every positive 
integer $m$. 
The $\ell_{p,q}$ mixed-norm clustering problem is formulated as the 
following mixed-integer program with a bi-convex objective function
and a Heaviside constraint:
\begin{equation} \label{eq:pq}
\begin{array}{ll}
\displaystyle{
\operatornamewithlimits{\mbox{\bf minimize}}_{\boldsymbol{c}, \, 
\boldsymbol{z}}
} & F_{p, q}(\boldsymbol{c}, \boldsymbol{z}) \triangleq \underbrace{\displaystyle{
\frac{1}{N}
} \, \displaystyle{
\sum_{s=1}^N
} \, \left[ \, \displaystyle{
\sum_{k=1}^K
} z_{sk} \, \| \, \xi^s - c^k \, \|_p^p \, \right]}_{\mbox{bi-convex}} 
\\ [0.2in]
\mbox{\bf s.t.} & \displaystyle{
\sum_{k=1}^K
} \, z_{sk} = 1, \hspace{0.25in} \forall \, s \, \in \, [ \, N \, ], \\ [0.2in]
& 
z_{sk} \, \leq \, \onebld_{[ \, 0,\infty )}\left( \, \displaystyle{
\min_{k^{\prime} \neq k}
} \, \| \, \xi^s - c^{k^{\prime}} \, \|_q - 
\| \, \xi^s - c^k \, \|_q \, \right), \hspace{0.25in} \forall (s,k) \, \in \, 
[ \, N \, ] \times [ \, K \, ] \\ [0.2in]
& z_{sk} \, \in \, \{ \, 0,1 \, \}, \hspace{0.25in} \forall (s,k) \, \in \, 
[ \, N \, ] \times [ \, K \, ],
\end{array}
\end{equation}
where $\| \, \bullet \, \|_{p}$ and $\| \, \bullet \, \|_{q}$ are the $\ell_p$ 
and $\ell_q$ vector
norms, respectively, and $\onebld_{[ \, 0,\infty )}( \bullet )$ is the 
Heaviside (i.e., indicator) function of the closed interval 
$[ 0, \infty )$; i.e., 
\[\onebld_{[ \, 0,\infty )}(t) = \left\{ \begin{array}{ll}
1 & \mbox{if $t \geq 0$} \\ [3pt]
0 & \mbox{if $t < 0$.}	
\end{array} \right.\]
The Heaviside constraint
requires every observation to be assigned to a nearest center in the
$\ell_q$-norm. Ties are allowed and are resolved by the binary assignment
variables.  The objective then evaluates the resulting partition using the
the $\ell_p$-norm (to avoid nonsmoothness of the objective, this norm is
raised to the $p$-th power). The introduced
optimization formalization in the form of a mixed-binary program provides
a solid foundation for a systematic treatment by a unified computational 
approach, enabling in particular generalizations to include constraints 
and the easy adoption of other clustering criteria. The mixed-norm clustering model \eqref{eq:pq} includes K-Means (i.e., $p = q = 2$) and K-Medians (i.e., $p = q = 1$) as special cases. When $p=q$, nearest-center assignment is already
consistent with minimization of the fitting objective.  Consequently, at the
level of problem formulation and solution, the Heaviside constraints are redundant 
and \eqref{eq:pq} reduces to
\begin{equation} \label{eq:p=q}
\begin{array}{ll}
\displaystyle{
\operatornamewithlimits{\mbox{\bf minimize}}_{\boldsymbol{c}, \, 
\boldsymbol{z}}
} & F_{p, p}(\boldsymbol{c}, \boldsymbol{z}) \triangleq \underbrace{\displaystyle{
\frac{1}{N}
} \, \displaystyle{
\sum_{s=1}^N
} \, \left[ \, \displaystyle{
\sum_{k=1}^K
} z_{sk} \, \| \, \xi^s - c^k \, \|_p^p \, \right]}_{\mbox{within-cluster
deviation (WCD$_p$)}} \\ [0.45in]
\mbox{\bf subject to} & \displaystyle{
\sum_{k=1}^K
} \, z_{sk} = 1, \hspace{0.25in} \forall \, s \, \in \, [ \, N \, ] \\ [0.2in]
& z_{sk} \, \in \, \{ \, 0,1 \, \}, 
\hspace{0.16in} \forall \, (s,k) \, \in \, [ \, N \, ] \, \times \, [ \, K \, ]. 
\end{array}
\end{equation}
This is a significant simplification from the mixed case with $p \neq q$. 
The matched-norm cases have been extensively studied as a mathematical programming 
problem \cite{bradley1999mathematical, bradley2000kplane}.  Early mathematical
programming formulations include the binary model for minimum within-group
sum of squares in \cite{rao1971cluster} and the concave-minimization treatment
of $K$-medians in \cite{bradley1997clustering}.  Subsequent work developed
exact and relaxation-based methods for $K$-means, including column generation
\cite{aloise2012improved} and semidefinite programming
\cite{peng2007approximating,piccialli2022sossdp}.  Bregman-divergence clustering 
models also broaden
the geometry used in clustering, but they retain a common divergence for
assignment and within-cluster evaluation \cite{banerjee2005clustering}. 

The mixed-norm case $p\ne q$ is structurally different and remains 
largely unexplored.
An assignment chosen by
$\ell_q$-proximity need not decrease the $\ell_p$ objective, while a 
center
update that reduces the $\ell_p$-loss can invalidate the $\ell_q$
nearest-center conditions.  Thus the descent mechanism underlying 
the usual
alternating method is lost, and the feasible assignment mapping changes
discontinuously with the centers. In \cite{herold2026clustering}, the 
authors study a general $(f,g)$-clustering 
problem in which the assignment criterion $f$ and the centering objective $g$ 
are allowed to differ,
making the framework closest in spirit to the $\ell_{p,q}$-clustering.
The analysis in the reference focuses on approximation guarantees for 
specific choices of $f$ and $g$
within the $p = q$ regime; the mixed-norm case $p \neq q$ is not considered.  
To the best of our knowledge, the specific
model \eqref{eq:pq} with $p\ne q$ is to date at best minimally studied and analyzed
rigorously, if at all. This gap is important both theoretically and computationally.

Integer programming (IP) offers a principled way to handle the general
mixed-norm clustering problem, 
but a direct solution of the full IP formulation can be prohibitively expensive.  
Even for the classical $K$-means model, sophisticated exact methods require 
specialized decomposition, bounding, and relaxation techniques
\cite{aloise2012improved,piccialli2022sossdp}; the mixed-norm constraints add a
further layer of combinatorial and discontinuous structure.  We address this
computational challenge through the novel methodology of progressive 
integer programming (PIP); this approach is built on 
the framework of Heaviside composite optimization that is best suited 
for decision-making problems governed by discontinuous comparisons
\cite{CuiLiuPang23-piecewise,HanCuiPang25}.  This line of research originated
in optimization models for individualized decision making
\cite{QiCuiLiuPang19} that has inspired the PIP methods for constrained Heaviside
programs and other challenging nonconvex optimization problems
\cite{FangLiuPang25,ZhangHanPang26,ZhengLiuWangPang26}. Instead of applying an 
off-the-shelf solver to solve the entire mixed-binary model directly, 
the progressive idea iteratively
solves a sequence of restricted subproblems in which a subset of the binary variables 
is adaptively fixed and updated based on a current candidate solution.  This simple
idea enables a
standard solver, such as {\sc Gurobi} \cite{gurobi}, to be effectively applied to 
compute solutions for otherwise intractable mixed IP models. Starting from a feasible 
solution, such as one obtained by an
alternating method, the PIP method generates non-worsening solutions
and may escape ``bad'' fixed points due to the combinatorial search with the sequence of restricted IPs solved.  We will show that, under the assumptions stated later, the solution obtained by PIP with an affordable computational cost is a strong local minimizer of \eqref{eq:pq}. This provides a stronger guarantee than the Nash property of the obtained solutions by an alternating algorithm even for the cases $p = q$. It is worth mentioning that a modified algorithm has been developed recently in \cite{li2025modified} to obtain a local optimal solution of K-means. In principle, 
the PIP scheme can be applied
to any combined pair of $(p,q)$ as long as there are effective IP algorithms
for solving the reduced subproblems in the PIP algorithm.
% that may be nonconvex with binary variables.  
Nevertheless, handicapped by the lack of such methods and available 
implementations for solving nonconvex mixed-binary programs, we put particular focus on a class of (\ref{eq:pq}) that 
can be effectively handled by state-of-the-art IP solvers, most prominent 
of which is {\sc gurobi} \cite{gurobi} that is the chosen solver to be
adopted in the PIP algorithm.  Supported by a local optimality theory, 
our encouraging numerical results demonstrate that 
the computed solutions by the PIP method are superior to those obtained by an 
alternating algorithm, among other comparisons.
We summarize the main contributions of this paper as follows:\vspace{0.05in}\\
\noindent $\bullet$ We introduce and systematically study a mixed-norm clustering model based on an integer programming approach, which, as mentioned, includes the classical models of K-means and K-medians as special cases. The independent choices of norms for cluster assignments and centroid updates offer flexibility in realistic applications and can achieve better clustering performance on noisy data.\vspace{0.05in}\\ 
$\bullet$ We develop a unified mixed IP-based method in the spirit of PIP to address the computational challenges for solving \eqref{eq:pq} at scale. Our
implementation focuses on the cases whose restricted subproblems can be solved
effectively by the state-of-the-art solver {\sc Gurobi}, while the
algorithmic framework itself is not tied to a particular solver.\vspace{0.05in}\\
$\bullet$ We derive necessary and sufficient conditions for local optimality of
\eqref{eq:pq}. More importantly, under some practical assumptions, we show that the solutions computed by the PIP method are local minimizers of the problem \eqref{eq:pq}. \vspace{0.05in}\\
% This is a stronger property than the Nash property of the solutions obtained 
% by the alternating procedure (for $p=q$).
$\bullet$ We conduct an extensive comparative computational study of the proposed PIP method with the alternating method for solving \eqref{eq:pq}.\vspace{0.05in} \\ The numerical results demonstrate the superior performance of the PIP method. Warm-started from the computed solution of its non-IP counterpart, the
PIP method is guaranteed to return a result no worse than the given solution. In particular, when the
former is trapped in a poor iterate that is not a fixed point of the algorithmic map, the IP’s combinatorial search escapes the point and produces strict improvement. This guarantee holds regardless of
the initialization of the non-IP method and is confirmed empirically across all benchmark settings.

The rest of the paper is organized as follows. We introduce the
principal ingredients of the PIP method in 
Section~\ref{sec:theory} for the mixed-norm $\ell_{p,q}$-problem
(\ref{eq:pq}), present a brief theory of the method, 
and establish a sharpened result for the matched-norm case
with $p = q$.  We design a practical implementation of the 
PIP algorithm for the case $q = 1$ in Section~\ref{sec: PIP-q=1}. 
Extensive numerical results are presented in 
Section~\ref{sec: numerical-results}. The paper ends with
a conclusion and some open research questions

\vspace{-0.1in}

\section{Partial IP Formulation and Optimality Theory}
\label{sec:theory}

This section develops the partial IP formulation and optimality 
theory for solving the mixed-norm clustering model \eqref{eq:pq} 
via the PIP method. Particularly, we will formulate an equivalent 
mixed-integer formulation to \eqref{eq:pq} and the corresponding 
restricted subproblems via the partial IP formulation;
more importantly, we will establish that a fixed point to the 
minimization of the partial IP problem is a locally optimal solution 
to \eqref{eq:pq}. The analysis distinguishes two notions of local
optimality.  An ordinary local minimizer is local in the product 
space of centers and binary assignments; since the binary points 
are isolated, this notion
effectively fixes the assignment near the reference point.  A
{\sl strong center-local minimizer} is optimal among all feasible
assignments whenever the centers are sufficiently close.  
The latter notion is stronger and is the principal guarantee 
obtained from a current iterate that itself is a global solution 
of a self-defined PIP subproblem. We call the latter a 
{\sl self-global} property.

The starting point of the IP approach for solving the
mixed-norm problem (\ref{eq:pq}) is to restrict the centers $c^k$ 
for which there is a constant $M > 0$ that bounds the norms 
$\max\{ \, \| \, \xi^s - c^k \|_p^p, \, \| \, \xi^s - c^k \, \|_q 
\, \}$ for all pairs $(s,k) \in [N] \times [K]$. We denote by 
${\cal C}$ a convex set of bounded centers (by the scalar $M$) 
that becomes part of the constraints in the following IP
formulation (\ref{eq:MIP}).  One example of the set ${\cal C}$ is
\[
{\cal C} \, = \, \displaystyle{
\prod_{k=1}^K 
} \ \displaystyle{
\prod_{j=1}^d
} \ \left[ \, \underline{\gamma}^k_j, \, \overline{\gamma}^k_j 
\, \right], 
\epc \mbox{for some constants $-\infty < \underline{\gamma}^k_j 
\, < \, \overline{\gamma}^k_j < +\infty$.}
\]
The componentwise or pair-dependent bounds may be used
computationally. In what follows,
we use a single big-M bound for simplicity of notation, in terms 
of which, we may linearize the 
product $t_{sk} \triangleq z_{sk} \, \| \, \xi^s - c^k \, \|_p^p$ 
in \eqref{eq:pq}
for each $(s, k) \in [N] \times [K]$.  This linearization is due 
to the equivalence:
\[ 
t_{sk} = z_{sk} \, \| \, \xi^s - c^k \, \|_p^p \ \Longleftrightarrow 
\ 0 \, \geq \, t_{sk} - \| \, \xi^s - c^k \, \|_p^p 
\, \geq \, -M \, ( 1 - z_{sk} ) \ \mbox{ and } \ 
M \, z_{sk} \, \geq \, t_{sk} \, \geq \, 0.
\]
Since $t_{sk}$ is being minimized, the right-hand conditions 
simplify to
\[
 t_{sk} - \| \, \xi^s - c^k \, \|_p^p \, \geq \, -M \, ( 1 - z_{sk} ) 
 \ \mbox{ and } \ t_{sk} \, \geq \, 0.
\]
With these auxiliary variables $t_{sk}$, we obtain 
a reformulation of (\ref{eq:pq}) as a mixed-integer 
$\ell_{p,q}$ norm-constrained program whose feasible set we 
denote ${\cal F}$:
\begin{equation}\label{eq:MIP}
% (\mathbf{P}_{p,1}^{\, \prime}) \left\{ \, 
\begin{array}{ll}
\underset{\boldsymbol{c} \, \in \, {\cal C}, \, \boldsymbol{z}, 
\, \boldsymbol{t}}{
\mbox{\bf minimize}} & \displaystyle{
\frac{1}{N}
} \, \displaystyle{
\sum_{s=1}^N
} \, \displaystyle{
\sum_{k=1}^K
} \, t_{sk} \\ [0.15in]
\mbox{\bf subject to} & \displaystyle{
\sum_{k=1}^K
} \, z_{sk} = 1, \hspace{0.25in} \forall \, s \, \in \, [ \, N \, ] 
\\ [0.2in]
\mbox{\bf and} & \left\{ \, \begin{array}{l}
\mbox{for all $(s,k) \, \in \, 
[ \, N \, ] \times [ \, K \, ]$:} \\ [0.1in]
\underbrace{t_{sk} \, \geq \, \| \, \xi^s - c^k \, \|_p^p - M(1 - z_{sk})}_{
\mbox{convex constraint; quadratic when $p = 2$}};
\epc t_{sk} \, \geq \, 0 \\ [0.35in]
\underbrace{\| \, \xi^s - c^k \, \|_q \, \leq \, 
\| \, \xi^s - c^{k^{\prime}} \, \|_q + M \, ( 1 - z_{sk})}_{
% \displaystyle{
% \sum_{j=1}^d
% } \, | \, \xi^s_j - c^k_j \, | - \displaystyle{
% \sum_{j=1}^d
% } \, | \, \xi^s_j - c^{k^{\prime}}_j \, | \, \leq \, M \, (1 - z_{sk})}_{
\mbox{nonconvex constraint; piecewise affine when $q = 1$}},
\quad \forall \, k^{\prime} \, \neq \, k \\ [0.35in]
z_{sk} \, \in \, \{ \, 0,1 \, \}.
\end{array} \right.
\end{array} 
\end{equation}
The equivalence of \eqref{eq:pq} and~\eqref{eq:MIP} follows from a
direct case analysis on $z_{sk} \in \{0,1\}$: when $z_{sk} = 1$, the
big-$M$ term vanishes and forces $t_{sk} = \|c^k - \xi^s\|_p^p$ at
optimality, while the assignment constraint reduces to the $\ell_q$ 
nearest-center
membership; when $z_{sk} = 0$, both constraints are inactive and 
we may choose $t_{sk} = 0$ which contributes nothing to the objective.

\gap

\noindent {\bf The challenges for solving (\ref{eq:MIP}):}
The constraints involving the $\ell_p$-norm are easier to handle due to convexity. Particularly, for $p = 1$, the convex constraints are essentially of the
linear kind, thus easy to handle; for $p = 2$, the convex quadratic
constraint in (\ref{eq:MIP}) is within the comfort reach of an IP solver 
like {\sc gurobi}. But it may still pose a challenge for problems with 
large $N$ ($K$ is typically of modest size).
A core challenge comes from the nonconvex constraints involving the $\ell_q$-norm. When $q = 2$, the nonconvex constraints in
(\ref{eq:MIP}) make it difficult for existing IP methods/solvers to 
handle effectively, even for low-dimensional problems.
When $q = 1$, the nonconvex constraints are piecewise affine.
While there are specialized IP methods
\cite{VielmaAhmedNemhauser10} designed to handle such constraints,
subsequently, we employ a convexification (in this case, linearization)
technique derived from difference-of-convex programming 
\cite{PangRazaAlvarado17}. 
Nevertheless, due to the co-existence of the convex and nonconvex
constraints and the binary variables, the practical limitation of
solving (\ref{eq:MIP}) by a straightforward application of an IP
solver is fully evident.  The situation becomes somewhat friendlier
in the special case of $p = q$ for which there is the 
equivalent formulation of (\ref{eq:p=q}):
\begin{equation}\label{eq:MIP p=q}
\begin{array}{ll}
\underset{\boldsymbol{c} \, \in \, {\cal C}; \, \boldsymbol{z}; 
\, \boldsymbol{t}}{\mbox{\bf minimize}} & \displaystyle{
\frac{1}{N}
} \, \displaystyle{
\sum_{s=1}^N
} \, \displaystyle{
\sum_{k=1}^K
} \ t_{sk} \\ [0.15in]
\mbox{\bf subject to} & \displaystyle{
\sum_{k=1}^K
} \, z_{sk} = 1, \hspace{0.25in} \forall \, s \, \in \, [ \, N \, ]
\\ [0.2in]
\mbox{\bf and} & \left\{ \begin{array}{l}
\mbox{for all $(s,k) \, \in \, [ \, N \, ] \times [ \, K \, ]$:} \\ [0.1in]
t_{sk} \, \geq \, \| \, \xi^s - c^k \, \|_p^p 
- M(1 - z_{sk}); \epc t_{sk} \, \geq \, 0 \\ [0.1in]
z_{sk} \, \in \, \{ \, 0,1 \, \},
\end{array} \right.
\end{array} 
\end{equation}
which is a convex mixed-binary program that is in principle directly
amenable to computational solution by an IP solver. Nevertheless,
as aforementioned, unless $p = 1$, the convex constraints could
remain a bottleneck for
the straightforward application of an IP solver to (\ref{eq:MIP p=q}) 
for large $N$.    

In summary, there
are many hurdles in solving the MIP (\ref{eq:MIP}) in a straightforward
way, due to nonconvexity, nonlinearity, and potentially
large scale of the problem.  This motivates the design of effective schemes to
alleviate the computational challenges in solving \eqref{eq:MIP}. One such scheme is
the progressive IP method that is based on solving partial IP problems
obtained by fixing some of the binary variables.  
% In what follows, we present the method for arbitrary pairs $(p,q)$ 
% but specialize it for practical implementation to cases where the 
% subproblems can be efficiently solved.  

\subsection{Partial IP models based on the softmax partition}

To reduce the computational burden of solving~\eqref{eq:MIP} with 
the full set of $NK$ binary variables $z_{sk}$ and $NK(K-1)$ constraints, we introduce a reduction 
of these variables as well as constraints by adaptively fixing 
some of them corresponding to a
given pair $\boldsymbol{\bar{y}} \triangleq 
\left( \boldsymbol{\bar{z}} \triangleq 
\{ \bar{z}_{sk} \}_{(s,k)=1}^{(N,K)}; \, 
\boldsymbol{\bar{c}} \triangleq \{ \bar{c}^k_j \}_{(j,k)=1}^{(d,K)} 
\, \right)$,
which along with the auxiliary tuple $\boldsymbol{\bar{t}} \triangleq 
\left( \{ \bar{t}_{sk} \}_{(s,k)=1}^{(N,K)} \right)$ with 
$\bar{t}_{sk} = \bar{z}_{sk} \, \| \, \xi^s - c^k \, \|_p^p$, is
feasible to (\ref{eq:MIP}). Consequently, the restricted subproblems 
leave the un-fixed binary 
variables as free variables to be solved.  
Occasionally, we will omit the $t_{sk}$-variables for simplicity when
they are clear from the context.  Informally, the rationale behind this fixing strategy is that most assignment variables (i.e., $z_{sk}$)  
are unambiguous if the current solution pair on hand is reasonably good. Therefore, adaptively fixing a subset of binary variables is sensible given that we design an effective correction mechanism to update the fixing set that corresponds to the ambiguous data points. The strict nearest and strict 
non-nearest index sets at $\bar{\boldsymbol c}$ are, respectively,
\begin{align}
 {\cal J}_{<}(\bar{\boldsymbol c})
 &\coloneqq
 \left\{(s,k) \in [ \, N \, ] \times [ \, K \, ] \, \mid \,
 \lVert\xi^s-\bar c^k\rVert_q
 <\min_{k'\ne k}\lVert\xi^s-\bar c^{k'}\rVert_q
 \right\}, \label{eq:Jless}\\
 {\cal J}_{>}(\bar{\boldsymbol c})
 &\coloneqq
 \left\{(s,k) \in [ \, N \, ] \times [ \, K \, ] \, \mid \,
 \lVert\xi^s-\bar c^k\rVert_q
 >\min_{k'\ne k}\lVert\xi^s-\bar c^{k'}\rVert_q
 \right\}. \label{eq:Jgreater}
\end{align}
The feasibility and the assignment equation imply that
\begin{equation}\label{eq:implications}
(s,k)\in{\cal J}_<(\bar{\boldsymbol c}) \ \Longrightarrow \ 
\bar z_{sk}=1, \qquad
(s,k)\in{\cal J}_>(\bar{\boldsymbol c}) \ \Longrightarrow \ 
\bar z_{sk}=0.
\end{equation}
For any $s$, there exists at most one index $k\in [K]$ such that $(s,k) \in \mathcal J_{<}(\bar c)$. Let ${\cal J}_1$, ${\cal J}_0$ be any two disjoint
index subsets of $[ \, N \, ] \times [ \, K \, ]$ such that
\begin{equation} \label{eq:index set inclusions}
{\cal J}_1 \, \subseteq \, {\cal J}_<(\boldsymbol{\bar{c}}) 
\epc \mbox{and} \epc
{\cal J}_0 \, \subseteq \, {\cal J}_>(\boldsymbol{\bar{c}}).
\end{equation}
Denote ${\cal J}_c \triangleq [ \, N \, ] \times \, [ \, K \, ] 
\, \setminus \,
\left( \, {\cal J}_1 \, \cup \, {\cal J}_0 \, \right).$ 
Therefore, we know that 
\[\left\{ \, (s,k)\, \mid \, 
\| \, \xi^s - \bar{c}^k \, \|_q \, = \, \displaystyle{
\min_{k^{\prime} \neq k}
} \, \| \, \xi^s - \bar{c}^{k^{\prime}} \, \|_q \, \right\} \subseteq {\cal J}_c.\]
The condition in the left-hand side of the above inclusion 
indicates that such a point $\xi^s$ achieves the minimum distance 
in the $\ell_q$-norm to multiple centroids and thus has an ambiguous 
assignment given $\bar{\bf c}$. Inspired by the implications in 
\eqref{eq:implications}, we construct 
a partial IP  problem of \eqref{eq:MIP} by fixing  the binary
variables in ${\cal J}_0$ and ${\cal J}_1$ as follows:
% \begin{equation}\label{eq:reduced MIP}
% % (\mathbf{P}_{p,1}^{\, \prime}) \left\{ \, 
% \begin{array}{ll}
% \underset{\boldsymbol{c} \, \in \, {\cal C}; \, \boldsymbol{z}; 
% \, \boldsymbol{t}}{\mbox{\bf minimize}} & \displaystyle{
% \frac{1}{N}
% } \, \displaystyle{
% \sum_{(s,k) \in {\cal J}_c}
% } \, t_{sk} + \displaystyle{
% \frac{1}{N}
% } \, \displaystyle{
% \sum_{(s,k) \in {\cal J}_1}
% } \ \| \, \xi^s - c^k \, \|_p^p \\ [0.2in]
% & \left( \, = \, \displaystyle{
% \frac{1}{N}
% } \, \displaystyle{
% \sum_{s=1}^N
% } \, \displaystyle{
% \sum_{k=1}^K
% } \, z_{sk} \, \| \, \xi^s - c^k \, \|_p^p \, = \, \displaystyle{
% \frac{1}{N}
% } \, \displaystyle{
% \sum_{s=1}^N
% } \, \displaystyle{
% \sum_{k=1}^K
% } \, t_{sk} \, \right) \\ [0.2in]
% \mbox{\bf subject to} & \displaystyle{
% \sum_{k=1}^K
% } \, z_{sk} = 1, \hspace{0.25in} \forall \, s \, \in \, [ \, N \, ] 
% \\ [0.2in]
% & \left\{ \begin{array}{l}
% \mbox{for all $(s,k) \, \in \, 
% [ \, N \, ] \, \times \, [ \, K \, ]$:} \\ [0.1in]
% t_{sk} \, \geq \, \| \, \xi^s - c^k \, \|_p^p - M(1 - z_{sk});
% \epc t_{sk} \, \geq \, 0 \\ [0.1in]
% \| \, \xi^s - c^k \, \|_q - \| \, \xi^s - c^{k^{\prime}} \, \|_q \, 
% \leq \, M \, (1 - z_{sk}), \quad \forall \, k^{\prime} \, \neq \, k 
% \end{array} \right\} \\ [0.4in]
% & \left\{ \begin{array}{l}
% z_{sk} \, \in \, \{ \, 0,1 \, \}, \epc (s,k) \, \in \, {\cal J}_c
% \\ [0.1in]
% \underbrace{z_{sk} \, = \, 1, \ (s,k) \, \in \, {\cal J}_1; 
% \ \mbox{ \bf and } \
% z_{sk} \, = \, 0, \ (s,k) \, \in \, {\cal J}_0}_{\mbox{restrictions
% occur here}}
% \end{array} \right\}
% \end{array} 
% \end{equation}
\begin{equation}\label{eq:reduced MIP}
\begin{array}{cll}
\displaystyle{
\operatornamewithlimits{\mbox{\bf minimize}}_{
\boldsymbol c\in{\cal C},\,\boldsymbol z,\,
\{ t_{sk}:(s,k)\in{\cal J}_c \}}
} \quad & \displaystyle{
\frac{1}{N}
} \, \left[ \, \displaystyle{
\sum_{(s,k)\in{\cal J}_c}t_{sk}
} + \displaystyle{
\sum_{(s,k)\in{\cal J}_1}
} \, \lVert\xi^s-c^k\rVert_p^p \, \right] & \\ [0.2in]
\displaystyle{
\operatornamewithlimits{\mbox{\bf subject to}}
} \quad & \displaystyle{
\sum_{k=1}^K
} \, z_{sk} \, = \, 1, & s \in [N] \\ [0.1in]
& t_{sk} \, \geq \, \lVert \, \xi^s-c^k \, \rVert_p^p -
M (1-z_{sk}), \quad t_{sk} \, \geq \, 0, & (s,k) \in {\cal J}_c 
\\ [0.1in]
& \lVert \, \xi^s-c^k \, \rVert_q - \lVert \, \xi^s-c^{k^{\prime}}
\, \rVert_q \, \leq \, M (1-z_{sk}), & (s,k) \in {\cal J}_c, \ 
k^{\prime} \neq k \\ [0.1in]
& z_{sk} \, \in \, \{0,1\}, & (s,k) \in {\cal J}_c \\ [0.1in]
& \lVert \, \xi^s-c^k \, \rVert_q \, \leq \, 
\lVert \, \xi^s-c^{k^{\prime}} \, \rVert_q,
& (s,k) \in {\cal J}_1, \ k^{\prime} \neq k \\ [0.1in]
& z_{sk} \, = \, 1, & (s,k) \in {\cal J}_1 \\ [0.1in]
\mbox{\bf and} & z_{sk} \, = \, 0, & (s,k) \in {\cal J}_0.
\end{array}
\end{equation}
It is clear that (\ref{eq:reduced MIP}) is a restriction of
(\ref{eq:MIP}) and contain the incumbent tuple
$(\bar{\boldsymbol c},\bar{\boldsymbol z},\bar{\boldsymbol t})$. It is worth mentioning that this feasibility property is essential, which guarantees that the optimal solution to \eqref{eq:reduced MIP} will not be worse than 
$(\bar{\boldsymbol c},\bar{\boldsymbol z},\bar{\boldsymbol t})$. 
The restrictions formalize the trust we have in the current candidate
pair $( \boldsymbol{\bar{c}},\boldsymbol{\bar{z}} )$.  Namely, based on 
this pair, we judge that
for any $(s,k) \in \mathcal{J}_1$, the sample $s$ is deemed likely to
belong to the cluster $k$, thus we fix the assignment $z_{sk} = 1$;
similarly, for any $(s,k) \in \mathcal{J}_0$, the sample $s$ is 
deemed unlikely to
belong to the cluster $k$, thus we fix the assignment $z_{sk} = 0$.
% The simplex constraint $\displaystyle \sum_{k=1}^K z_{sk} = 1$ then forces 
% $z_{sk^{\prime}} = 0$ for all $k^{\prime} \neq k$, so fixing a single pair 
% in $\mathcal{J}^+$ implicitly determines all remaining assignments for 
% data point $s$.    
Binary variables are retained only for ambiguous pairs 
$(s,k) \in \mathcal{J}_{\, c}$  by the current information. 
  
A reasonable way to construct the index sets
${\cal J}_{\, 0}$ and ${\cal J}_1$ satisfying \eqref{eq:index set inclusions} is via the softmax-based probabilities
for a given scalar $\tau > 0$:
\begin{equation}\label{eq:softmax}
\rho_{sk}(\tau) \, \triangleq \, \displaystyle{
\frac{\exp\!\left(-\|\xi^s - \bar{c}^k\|_q \,/\, \tau\right)}
{\displaystyle\sum_{k^{\prime}=1}^{K}
\exp\!\left(-\|\xi^s - \bar{c}^{k^{\prime}}\|_q \,/\, \tau\right)}
}, \qquad \forall\, (s,k) \, \in \, [ \, N \, ] \, \times \, 
[ \, K \, ].
\end{equation}
Using confidence threshold $\displaystyle\frac{1}{2}$ and minimum 
viability threshold $\displaystyle \frac{1}{K}$,
we define the confident sets
\begin{equation} \label{eq:J0}
\mathcal{J}^{>}_{\tau}(\bar{\boldsymbol c}) \, \triangleq \,
\left\{ \, (s,k)\, \mid \, \rho_{sk}(\tau) \, > \, \thalf 
\, \right\} \ \mbox{ and } \
% \label{eq:Jplus} \\[4pt]
\mathcal{J}^{\, <}_{\tau}(\bar{\boldsymbol c}) \, \triangleq \,
\left\{ \, (s,k)\, \mid \, \rho_{sk}(\tau) \, < \, \displaystyle{
\frac{1}{K}
} 
\, \right\}.
\end{equation}
The proposition below shows that these two sets can play the role 
of $\mathcal{J}_1$ and $\mathcal{J}_0$, respectively 
( i.e., satisfying 
the inclusion requirements in (\ref{eq:index set inclusions})). 

\begin{proposition}\label{pr:feasibility} \rm
Let $( \boldsymbol{\bar{c}},\boldsymbol{\bar{z}},
\boldsymbol{\bar{t}} )$ 
be a feasible triplet of (\ref{eq:MIP}). For any $\tau > 0$, 
it holds that
$\mathcal{J}^{>}_\tau(\bar{\boldsymbol c}) \subseteq 
{\cal J}_<(\boldsymbol{\bar{c}})$ and
$\mathcal{J}^<_\tau(\bar{\boldsymbol c}) \subseteq 
{\cal J}_>(\boldsymbol{\bar{c}})$. For $K > 2$, we also have
% can replace $\mathcal{J}^{>}_\tau(\bar{\boldsymbol c})$ by 
$\mathcal{J}^{\geq}_{\tau}(\bar{\boldsymbol c}) \, \triangleq \,
\left\{ \, (s,k) \, \mid \, \rho_{sk}(\tau) \, \geq \, 
\thalf \, \right\} \subseteq {\cal J}_{<}(\boldsymbol{\bar{c}})$. 
\end{proposition}

\begin{proof}  We prove the last inclusion
$\mathcal{J}^{\geq}_{\tau}(\bar{\boldsymbol c}) \subseteq 
{\cal J}_{<}(\boldsymbol{\bar{c}})$; a simple modification of
the proof also yields the first one 
$\mathcal{J}^>_{\tau}(\bar{\boldsymbol c}) \subseteq 
{\cal J}_{<}(\boldsymbol{\bar{c}})$.  
Suppose $K > 2$ and $\rho_{sk}(\tau) \geq 1/2$.  
By definition~\eqref{eq:softmax},
\begin{equation*}
    \exp\!\left(\frac{-\|\xi^s-\bar{c}^k\|_q}{\tau}\right) 
    \, \geq \, 
    \frac{1}{2} \, \sum_{k^{\prime}=1}^{K}
    \exp\!\left(\frac{-\|\xi^s-\bar{c}^{k^{\prime}}\|_q}{\tau}
    \right).
\end{equation*}
Subtracting $\displaystyle{ 
\frac{1}{2}
} \, \exp\left( \frac{-\|\xi^s-\bar{c}^k\|_q}{\tau} \right)$ from 
both sides gives
\begin{equation*}
    \exp\!\left(-\frac{\|\xi^s-\bar{c}^k\|_q}{\tau}\right)
    \, \geq \,
    \sum_{k^{\prime} \neq k}
    \exp\!\left(-\frac{\|\xi^s-\bar{c}^{k^{\prime}}\|_q}{\tau}\right)
    \, > \, 
    \max_{k^{\prime} \neq k}
    \exp\!\left(-\frac{\|\xi^s-\bar{c}^{k^{\prime}} \|_q}{\tau}\right).
\end{equation*}
which clearly yields
% Taking the natural logarithm and multiplying by $-\tau<0$ reverses the
% inequality, yielding
% $\|\xi^s-\bar{c}^k\|_1 < \min_{k'\neq k}\|\xi^s-\bar{c}^{k'}\|_1$,
% or equivalently
\begin{equation*}
\displaystyle{
\min_{k^{\prime} \neq k}
} \, \| \, \xi^s-\bar{c}^{k^{\prime}} \|_q - 
\| \, \xi^s-\bar{c}^k \, \|_q \, > \, 0.
\end{equation*}
Thus, $(s,k) \in {\cal J}_<(\boldsymbol{\bar{c}})$.  
Similarly, to prove
$\mathcal{J}^<_\tau(\bar{\boldsymbol c}) \subseteq 
{\cal J}_>(\boldsymbol{\bar{c}})$, suppose that 
$\rho_{sk}(\tau) < \displaystyle{
\frac{1}{K}
}$.  Since $\displaystyle{
\sum_{k^{\prime}=1}^K
} \, \rho_{sk^{\prime}}(\tau) = 1$; we deduce
$\displaystyle{
\max_{k^{\prime} \neq k}
} \, \rho_{sk^{\prime}}(\tau) \geq \displaystyle{
\frac{1}{K}
} > \rho_{sk}(\tau)$, or equivalently, 
$\| \, \xi^s - \bar{c}^k \ \|_q \, > \, \displaystyle{
\min_{k^{\prime} \neq k}
} \, \| \, \xi^s - \bar{c}^{k^{\prime}} \, ]\|_q$.  Thus 
$(s,k) \in {\cal J}_>(\boldsymbol{\bar{c}})$.
\end{proof}

Next, we establish a strong local connection between the 
partial IP (\ref{eq:reduced MIP}) and the original clustering
problem (\ref{eq:pq}). We call
$(\bar{\boldsymbol c},\bar{\boldsymbol z})$ a 
{\sl strong center-local minimizer}
if a neighborhood ${\cal N}$ of $\bar{\boldsymbol c}$ 
exists such
that
\[
 F_{p, q}(\bar{\boldsymbol c},\bar{\boldsymbol z})
 \leq F_{p, q}(\boldsymbol c,\boldsymbol z)
\]
for every feasible $(\boldsymbol c,\boldsymbol z)$ of~\eqref{eq:pq} with
$\boldsymbol c\in{\cal N}$, with no proximity restriction imposed on
$\boldsymbol z$.
% in terms of a fixed-point property of a feasible pair 
% $( \boldsymbol{\bar{c}},\boldsymbol{\bar{z}} )$ of (\ref{eq:pq})
% that defines the reduced IP (\ref{eq:reduced MIP}). 

\begin{theorem} \label{th:fixed-point localmin} \rm
Let $(\boldsymbol{\bar{c}},\boldsymbol{\bar{z}})$ 
be a feasible pair for (\ref{eq:pq}) and 
$\boldsymbol{\bar{t}} \in \mathbb{R}^{NK}$ with 
$\bar t_{sk}=\bar z_{sk}\lVert\xi^s-\bar c^k\rVert_p^p$.  
Among the following three statements, it holds that 
(A) $\Rightarrow$ (B) $\Rightarrow$ (C):

\noindent (A) the triplet
$( \boldsymbol{\bar{c}},\boldsymbol{\bar{z}},\boldsymbol{\bar{t}} )$ 
is (globally) optimal for (\ref{eq:reduced MIP}) for some pair
of disjoint index sets ${\cal J}_1$ and ${\cal J}_0$ 
% defined by the pair 
% $( \boldsymbol{\bar{c}},\boldsymbol{\bar{z}} )$
satisfying the inclusions (\ref{eq:index set inclusions});

% \gap

\noindent (B) the pair
$( \boldsymbol{\bar{c}},\boldsymbol{\bar{z}} )$ is a strong center 
local minimizer of (\ref{eq:pq}) for some neighborhood 
$\mathcal{N}$ of $\boldsymbol{\bar{c}}$;

% \gap

\noindent (C) the pair 
$( \boldsymbol{\bar{c}},\boldsymbol{\bar{z}} )$ 
is an optimal solution of (\ref{eq:pq}) over all 
$( \boldsymbol{c},\boldsymbol{z} )$ sufficiently close to 
$( \boldsymbol{\bar{c}},\boldsymbol{\bar{z}} )$.
\end{theorem}

\begin{proof}
The implication (B) $\Rightarrow$ (C) is immediate. Next, we show that (A) $\Rightarrow$ (B).  Let the neighborhood ${\cal N}$ be 
chosen such that for all $s \in [ N ]$, 
$(k,k^{\prime}) \in [ K ] \times [ K ]$, and $\boldsymbol{c} \in {\cal N}$,
\begin{equation} \label{eq:continuity}
\| \, \xi^s - \bar{c}^k \, \|_q \, < \, 
\| \, \xi^s - \bar{c}^{k^{\prime}} \, \|_q \, \ \Rightarrow \, 
\| \, \xi^s - c^k \, \|_q \, < \, \| \, \xi^s - c^{k^{\prime}} \, \|_q.
\end{equation}
Let $( \boldsymbol{c},\boldsymbol{z} ) \in {\cal N}
\times \{ \, 0,1 \, \}^{NK}$ be an arbitrary feasible pair to
(\ref{eq:pq}).  We need to show
\begin{equation} \label{eq:double sum obj}
F_{p, q}(\boldsymbol{c}, \boldsymbol{z}) \geq F_{p, q}(\boldsymbol{\bar{c}}, \boldsymbol{\bar{z}}).
\end{equation}
Given that
\[
\displaystyle{
\frac{1}{N}
} \, \displaystyle{
\sum_{(s,k) \in {\cal J}_c}
} \, t^{\, \prime}_{sk} + \displaystyle{
\frac{1}{N}
} \, \displaystyle{
\sum_{(s,k) \in {\cal J}_1}
} \, \| \, \xi^s - (c^{\prime})^{k} \, \|_p^p \, \geq \, 
\displaystyle{
\frac{1}{N}
} \, \displaystyle{
\sum_{(s,k) \in {\cal J}_c}
} \, \bar{t}_{sk} + \displaystyle{
\frac{1}{N}
} \, \displaystyle{
\sum_{(s,k) \in {\cal J}_1}
} \, \| \, \xi^s - \bar{c}^k \, \|_p^p
\]
for all $( \boldsymbol{c}^{\prime},\boldsymbol{z}^{\prime},
\boldsymbol{t}^{\, \prime} )$ feasible to (\ref{eq:reduced MIP}), 
it suffices to show that 
$( \boldsymbol{c},\boldsymbol{z},\boldsymbol{t} )$
is feasible to (\ref{eq:reduced MIP}), where 
$t_{sk} \triangleq z_{sk} \| \, \xi^s - c^k \|_p^p$.  In turn, 
it suffices to show that $z_{sk} = 1$ for all $(s,k) \in {\cal J}_1$ 
and $z_{sk} = 0$ for all $(s,k) \in {\cal J}_0$.  Let $(s,k) \in {\cal J}_1$.
Then $\| \, \xi^s - \bar{c}^k \, \|_q < \displaystyle{
\min_{k^{\, \prime} \neq k}
} \, \| \, \xi^s - \bar{c}^{k^{\, \prime}} \, \|_q$.  By (\ref{eq:continuity}),
it follows that $\| \, \xi^s - c^k \, \|_q < \displaystyle{
\min_{k^{\, \prime} \neq k}
} \, \| \, \xi^s - c^{k^{\, \prime}} \, \|_q$.  Hence, 
$z_{sk} = 1$ by (\ref{eq:implications}).  The claim for
$(s,k) \in {\cal J}_0$ follows a similar argument.
\end{proof}

% Next, we will show that 
Shown below, the converse (C) $\Rightarrow$ (A) 
% in Theorem \ref{th:fixed-point localmin} 
if an additional sample-wise unique center assumption holds.

\begin{assumption} \rm
\label{ass: sample-center-unique}
 For a feasible pair  $(\boldsymbol{\bar{c}},\boldsymbol{\bar{z}})$ for (\ref{eq:pq}) with $\bar t_{sk}=\bar z_{sk}\lVert\xi^s-\bar c^k\rVert_p^p$, the index set $\displaystyle{
\operatornamewithlimits{\mbox{\bf argmin}}_{1 \leq k^{\prime} \leq K}
} \, \| \, \xi^s - \bar{c}^{k^{\prime}} \, \|_q$ is a singleton 
for all $s \in [N]$.
\end{assumption}

\begin{theorem} \rm
Suppose that a feasible pair  $(\boldsymbol{\bar{c}},\boldsymbol{\bar{z}})$ for (\ref{eq:pq}) satisfies Assumption~\ref{ass: sample-center-unique}, Then the 
statements (A) and (C) in Theorem \ref{th:fixed-point localmin} 
are equivalent.
\end{theorem}

\begin{proof}
\noindent It is sufficient to prove that (C) $\Rightarrow$ (A) 
under Assumption \ref{ass: sample-center-unique}.
Let ${\cal N}_c \times {\cal N}_z$ be a
neighborhood of the pair $(\boldsymbol{\bar{c}},\boldsymbol{\bar{z}})$
within which this pair is optimal to the problem (\ref{eq:pq}).  
% Without loss of generality, we may assume that $z = \bar{z}$ for any
% $z \in {\cal N}_z$. 
Define
\[ \begin{array}{l}
{\cal J}_1 \, \triangleq \, \left\{ \, (s,k)\, \mid \, 
\| \, \xi^s - \bar{c}^k \, \|_q \, < \, \displaystyle{
\min_{k^{\prime} \neq k}
} \, \| \, \xi^s - \bar{c}^{k^{\prime}} \, \|_q \, \right\} \\ [0.15in]
{\cal J}_0 \, \triangleq \, \left\{ \, (s,k) \,: \, 
\| \, \xi^s - \bar{c}^k \, \|_q \, > \, \displaystyle{
\min_{k^{\prime} \neq k}
} \, \| \, \xi^s - \bar{c}^{k^{\prime}} \, \|_q \, \right\}.
\end{array} \]
The singleton assumption implies the following:  

\gap

\noindent (i) ${\cal J}_c = \left\{ \, (s,k) \, \mid \, 
\| \, \xi^s - \bar{c}^k \, \|_q \, = \, \displaystyle{
\min_{k^{\prime} \neq k}
} \, \| \, \xi^s - \bar{c}^{k^{\prime}} \, \|_q \, \right\}$ is empty; 

\gap

\noindent (ii) ${\cal J}_1 = \left\{ \, (s,k)\, \mid \, 
\bar{z}_{sk} = 1 \, \right\}$ and
${\cal J}_0 = \left\{ \, (s,k) \, \mid \, \bar{z}_{sk} = 0 \, 
\right\}$; 
thus ${\cal J}_1$ and ${\cal J}_0$ partition $[ N ] \times [ K ]$; 
hence any pair $( \boldsymbol{c},\boldsymbol{z} )$ that is feasible 
to (\ref{eq:reduced MIP}) must have 
$\boldsymbol{z} = \boldsymbol{\bar{z}}$; and 

\gap

\noindent (iii) the neighborhood ${\cal N}_c$ can be chosen 
such that for all
$\boldsymbol{c} \in {\cal N}_c$,
\[ \begin{array}{l}
(s,k) \, \in \, {\cal J}_1 \ \Longleftrightarrow \ 
\| \, \xi^s - c^k \, \|_q \, < \, \displaystyle{
\min_{k^{\prime} \neq k}
} \, \| \, \xi^s - c^{k^{\prime}} \, \|_q \\ [0.1in]
(s,k) \, \in \, {\cal J}_0 \ \Longleftrightarrow \ 
\| \, \xi^s - c^k \, \|_q \, > \, \displaystyle{
\min_{k^{\prime} \neq k}
} \, \| \, \xi^s - c^{k^{\prime}} \, \|_q
\end{array} \]
Let $( \boldsymbol{c},\boldsymbol{z} )$ be an arbitrary feasible pair of
(\ref{eq:reduced MIP}).  Then $\boldsymbol{z} = \boldsymbol{\bar{z}}$.
Moreover, it is not difficult to show that for any scalar $\tau \in ( 0,1 )$
sufficiently small, the pair
$( \boldsymbol{c}^{\tau},\boldsymbol{\bar{z}} )$ is feasible to
(\ref{eq:pq}), where $\boldsymbol{c}^{\tau} \triangleq 
\boldsymbol{\bar{c}} + \tau ( \boldsymbol{c} - \boldsymbol{\bar{c}} )$.
Consequently, it follows, by the local optimality of the pair
$(\boldsymbol{\bar{c}},\boldsymbol{\bar{z}})$ to (\ref{eq:pq}), that
\[
\displaystyle{
\frac{1}{N}
} \, \displaystyle{
\sum_{s=1}^N
} \, \left[ \, \displaystyle{
\sum_{k=1}^K
} \, \bar{z}_{sk} \, \| \, \xi^s - c^{\tau;k} \, \|_p^p \, \right] 
\, \geq \, \displaystyle{
\frac{1}{N}
} \, \displaystyle{
\sum_{s=1}^N
} \, \left[ \, \displaystyle{
\sum_{k=1}^K
} \, \bar{z}_{sk} \, \| \, \xi^s - \bar{c}^k \, \|_p^p \, \right] 
\]
which implies, by the convexity of 
$\| \, \xi^s - \bullet \, \|_p^p$,
\[
\displaystyle{
\frac{1}{N}
} \, \displaystyle{
\sum_{s=1}^N
} \, \left[ \, \displaystyle{
\sum_{k=1}^K
} \, \bar{z}_{sk} \, \| \, \xi^s - c^k \, \|_p^p \, \right] 
\, \geq \, \displaystyle{
\frac{1}{N}
} \, \displaystyle{
\sum_{s=1}^N
} \, \left[ \, \displaystyle{
\sum_{k=1}^K
} \, \bar{z}_{sk} \, \| \, \xi^s - \bar{c}^k \, \|_p^p \, \right],
\]
establishing the statement (A).
\end{proof}

In the following, we construct a simple counterexample to show 
that statement (C) need not imply statement (A) in 
Theorem~\ref{th:fixed-point localmin} without 
Assumption~\ref{ass: sample-center-unique}. 
\begin{example} \rm
Let $p = q = 1$, $d=1$, $K=2$, $N=3$, ${\cal C}=[-3,3]^2$, 
$\xi^1=-2$, $\xi^2=0$, and $\xi^3=1$.
Consider $\bar c^1=-1$, $\bar c^2=1$, and
$\bar{\boldsymbol z}
= \begin{pmatrix}1&0\\1&0\\0&1\end{pmatrix}$.
The cluster assignment for the second sample $\xi^2$ is tied, 
whereas the other two assignments are strict.  The
objective value at $(\bar{\boldsymbol c}, \bar{\boldsymbol z})$ 
is
\[
F_{1,1}(\bar{\boldsymbol c},\bar{\boldsymbol z}) \, = \,
\frac{1}{3}(1+1+0)=\frac{2}{3}.
\]
For $c^1\in(-2,0)$ and $c^2\in(0,2)$, every feasible pair 
with assignment $\bar{\boldsymbol z}$ satisfies
\[
 F_{1,1}({\boldsymbol c},\bar{\boldsymbol z})
 =\frac13\bigl(|-2-c^1|+|c^1|+|1-c^2|\bigr)
 =\frac13\bigl(2+|1-c^2|\bigr)\geq\frac23.
\]
Moreover, a ball of radius less than $1/2$ around the binary vector
$\bar{\boldsymbol z}$ contains no other binary vector.  Hence, 
the pair $(\bar{\boldsymbol c}, \bar{\boldsymbol z})$ is a local minimizer of \eqref{eq:pq}. Note that 
\[
{\cal J}_<(\bar{\boldsymbol c}) =\{(1,1),(3,2)\},
 \qquad
 {\cal J}_>(\bar{\boldsymbol c}) =\{(1,2),(3,1)\}.
\]
Take $\widehat c^{\, 1}=-\frac32$, 
$\widehat c^{\, 2}=\frac12$, and $\widehat{\boldsymbol z}
=\begin{pmatrix}1&0\\0&1\\0&1\end{pmatrix}$,
which implies that 
$\widehat{\boldsymbol t}
=\begin{pmatrix}1/2&0\\0&1/2\\0&1/2\end{pmatrix}$.
For this pair 
$(\widehat{\boldsymbol c}, \widehat{\boldsymbol z})$, 
each sample is assigned to its unique nearest center, and the 
assignment respects every possible admissible fixing because 
it agrees with the four pairs in ${\cal J}_<(\bar{\boldsymbol c})$
and ${\cal J}_>(\bar{\boldsymbol c})$.  Thus it is feasible
for every reduced
problem that can be generated at the reference pair 
$(\bar{\boldsymbol c},\bar{\boldsymbol z})$. The objective value at 
$(\widehat{\boldsymbol c}, \widehat{\boldsymbol z})$ is
\[
 F_{1, 1}(\widehat{\boldsymbol c},\widehat{\boldsymbol z})
=\frac13\left(\frac12+\frac12+\frac12\right)
 =\frac12<\frac23.
\]
This implies that (A) fails.  \hfill $\Box$
\end{example}

\subsection{The matched norm case ($p = q$)}
\label{subsec:matched norm}

Traditionally,
a standard procedure for solving the $\ell_{p,q}$-problem 
(\ref{eq:p=q}) with $p = q$
is to alternate between the optimization of 
the $\boldsymbol{z}$ variable
(the clustering step) and the $\boldsymbol{c}$ variable 
(the centering
step).  In the former step, each point is assigned to its
$\ell_p$ nearest center, while the latter step updates each 
center to 
its $\ell_p$ optimal location within the cluster.  Due to three reasons:
(i) each step is nonincreasing in the objective; (ii) the number 
of distinct partitions (i.e., the $z_{sk}$ variables) is finite, 
and (iii) each updated center is unique given the partition, the 
procedure always terminates and the limit is a fixed point of the 
alternating map, at which neither step alone can reduce the 
objective.  More precisely, at termination, we obtain 
an {\sl equilibrium solution} of the Nash kind
of the bi-convex clustering problem (\ref{eq:p=q}); this 
is a pair 
$( \boldsymbol{\bar{c}},\boldsymbol{\bar{z}} )$ where
\[ \begin{array}{lll}
\boldsymbol{\bar{z}} & \triangleq & \left( \, z_{sk}\, \mid \, 
(s,k) \in [ N ] \times [ K ] \, \right) \\ [0.1in]
& \in & \displaystyle{
\operatornamewithlimits{\mbox{\bf argmin}}_{\boldsymbol{z} 
\, \in \, \{ 0,1 \}^{KN}}
} \, \underbrace{\displaystyle{
\sum_{s=1}^N
} \, \left[ \, \displaystyle{
\sum_{k=1}^K
} \,  z_{sk} \, \| \, \xi^s - \bar{c}^k \, \|_p^p \, \right]}_{\mbox{
assign clusters given centers fixed at $\boldsymbol{\bar{c}}$}} \,
\mbox{ \bf subject to} \epc \displaystyle{
\sum_{k=1}^K
} \, z_{sk} = 1, \epc \forall \, s \, \in \, [ \, N \, ]
\end{array} \]
% & & \epc \mbox{ \bf and } \epc z_{sk} \, \left\{ \begin{array}{ll}
% = \, 1 & \mbox{if $k \, \in \, \displaystyle{
% \operatornamewithlimits{\mbox{\bf argmin}}_{1 \leq j \leq K}
% } \, \| \, \xi^s - \bar{c}^j \, \|_q$} \\ [0.15in]
% = \, 0 & \mbox{otherwise.}
% \end{array} \right\} \ \forall \, (s,k) \, \in \, 
% [ \, N \, ] \, \times \, [ \, K \, ] \\ [0.2in]
and 
\[
\boldsymbol{\bar{c}} \, \triangleq \, \left( \, c^k_j\, \mid \, 
(j,k) \in [ d ] \times [ K ] \, \right)  \, \in \, \displaystyle{
\operatornamewithlimits{\mbox{\bf argmin}}_{\boldsymbol{c} \, \in \,
\mathbb{R}^{Kd}}
} \ \underbrace{\displaystyle{
\sum_{k=1}^K
} \, \left[ \, \displaystyle{
\sum_{s=1}^N
} \, \bar{z}_{sk} \, \| \, \xi^s - c^k \, \|_p^p\, \right]}_{\mbox{update
centers given clusters fixed by $\boldsymbol{\bar{z}}$}}.
\]
Clearly, there is no guarantee that the resulting pair
$( \boldsymbol{\bar{c}},\boldsymbol{\bar{z}} )$ will optimize 
the overall objective of (\ref{eq:pq}) globally.  Nevertheless, 
such a pair must optimize the objective locally.  Indeed, any pair 
$( \boldsymbol{c},\boldsymbol{z} )$ that is sufficiently
close to $( \boldsymbol{\bar{c}},\boldsymbol{\bar{z}} )$ must have
$\boldsymbol{z} = \boldsymbol{\bar{z}}$.  Thus a Nash equilibrium 
solution of (\ref{eq:pq}) must be a local minimizer of 
(\ref{eq:pq}).  Yet a local minimizer may not be global because
joint changes to both the centers
and the assignments may still decrease the objective.
As a result, the final solution obtained by an alternating 
method, which depends heavily on the starting point, can be 
far from globally
optimal.  We remark that while the alternating method remains 
in principle implementable
for $p \neq q$ based on the formulation (\ref{eq:pq}), 
convergence does not seem likely; an immediate reason is that 
decrease of the objective
is not guaranteed due to the Heaviside constraint.   

The question of whether alternating K-means converges to a local
optimum has been studied from an optimization perspective since
\cite{selim1984kmeans}, who characterized the local optimality
conditions for K-means type algorithms and proved finite
convergence of the alternating procedure.
Whether the algorithm always reaches a local minimum point 
remained open until recently.  The reference
\cite{li2025modified} establish conditions under which K-means
terminates at a local minimum and propose a modified algorithm
that guarantees local optimality in all cases, effectively
closing this question for the $p = q = 2$ setting.
For the mixed-norm case $p \neq q$, no analogous result exists:
the decoupling of assignment and quality norms means that the
structural conditions underlying the alternating-heuristic
modifications do not carry over, and the convergence question, 
let alone local optimality, remains entirely open.
The last reference, restricted to the K-means problem, addresses 
the question through exact 
integer programming, which bypasses the structural limitations 
of heuristic methods and
directly certifies local optimality as a consequence of solving
an IP subproblem to global optimality.

In what follows, we establish a sharpened result for
the $\ell_{p,p}$ problem based on the
simplified MIP formulation (\ref{eq:p=q}), which covers both the 
K-Means and K-Medians.  The simplified formulation permits removal 
of the nearest-center
constraints for the free assignment variables in the subproblems
(\ref{eq:reduced MIP}), but requires a completeness condition on
the zero fixings. Specifically, let
\begin{equation}\label{eq:matched index sets}
{\cal J}_1 \,  {=} \, 
{\cal J}_< (\bar{\boldsymbol c}), \qquad
{\cal J}_0 \, = \, {\cal J}_> (\bar{\boldsymbol c}), \qquad
{\cal J}_c \, = \, [ \, N \, ] \, \times \, [ \, K \, ]
\, \setminus \, \left( \, {\cal J}_1 \, \cup \, 
{\cal J}_0 \, \right).
\end{equation}
Note the particular choices of the two sets ${\cal J}_1$ and
${\cal J}_0$ of fixing the binary variables, whose
% Here, we choose ${\cal J}_0={\cal J}_> (\bar{\boldsymbol c})$ 
% to fix the predicted zeros. 
advantage will be clear momentarily.
The simplified restricted subproblem of \eqref{eq:reduced MIP}
has the following form:
\begin{equation}\label{eq:reduced MIP p=q}
\begin{array}{cll}
\displaystyle{
\operatornamewithlimits{\mbox{\bf minimize}}_{\boldsymbol c \in
{\cal C}, \, \boldsymbol z,\, \{ t_{sk}:(s,k)\in{\cal J}_c \}}
} & \displaystyle{
\frac{1}{N}
} \, \left[ \, \displaystyle{
\sum_{(s,k)\in{\cal J}_c}t_{sk}
} + \displaystyle{
\sum_{(s,k)\in{\cal J}_1}
} \, \lVert \, \xi^s-c^k \, \rVert_p^p \, \right] & \\ [0.2in]
\mbox{\bf subject to} & \displaystyle{
\sum_{k=1}^K
} \, z_{sk} \, = 1, & s \in [N] \\ [0.1in]
& t_{sk} \, \geq \, \lVert \, \xi^s-c^k \, \rVert_p^p -
M ( 1-z_{sk} ),\quad t_{sk} \, \geq \, 0, & (s,k) \in {\cal J}_c
\\ [0.2in]
& z_{sk} \, \in \, \{0,1\}, & (s,k) \in {\cal J}_c \\ [0.1in]
& \lVert \, \xi^s-c^k \, \rVert_p \, \leq \, 
\lVert \, \xi^s-c^{k^{\prime}} \, \rVert_p, 
& (s,k) \in {\cal J}_1,\ k^{\prime} \neq \, k \\ [0.1in]
& z_{sk} \, = \, 1, & (s,k) \in {\cal J}_1 \\ [0.1in]
\mbox{\bf and} & z_{sk} \, = \, 0, & (s,k) \in {\cal J}_0.
\end{array}
\end{equation}
% This problem is the computational cornerstone of the overall 
% algorithm to be presented later.
Unlike~\eqref{eq:reduced MIP}, the
formulation~\eqref{eq:reduced MIP p=q} does
not explicitly enforce nearest-center membership for   free 
binary variables with index $(s,k) \in \mathcal J_c$. The two 
index sets $\mathcal J_1$ and $\mathcal J_0$ are chosen equal 
to $\mathcal J_<(\bar {\boldsymbol c}) $ and $\mathcal J_>(\bar {\boldsymbol c})$ respectively in the above formulation for two 
reasons.   
% Its global solutions therefore need not be feasible 
% for~\eqref{eq:pq} far from the reference point. 
First, under the specific choice  ${\cal J}_0 \, = \, {\cal J}_> (\bar{\boldsymbol c})$ in (\ref{eq:matched index sets}), 
% the specific choice of 
% ${\cal J}_0={\cal J}_> (\bar{\boldsymbol c})$, 
we can obtain a 
characterization of the strong center-local optimality of 
\eqref{eq:p=q} regarding to problem \eqref{eq:reduced MIP p=q} in Theorem \ref{th:reduced p=q} for the $\ell_{p,p}$ problem without the sample-wise unique center 
assumption (i.e., Assumption \ref{ass: sample-center-unique}). Second, with $z_{sk}=0$ for all $(s,k) \in {\cal J}_0 \, = \, {\cal J}_> (\bar{\boldsymbol c})$, we have $z_{sk} =1$ for all $(s,k) \in \mathcal J_<(\bar{\boldsymbol{c}})$. Thus, without the nearest-center membership for free binary variables, we must set $\mathcal J_1 = \mathcal J_<(\bar {\boldsymbol c})$ in problem \eqref{eq:reduced MIP p=q} in order to keep it as a restricted subproblem to \eqref{eq:pq}.    
For clarity in the result below, we call the pair 
$(\bar{\boldsymbol c},\bar{\boldsymbol z})$ where 
$\bar{\boldsymbol{z}}$ is feasible to \eqref{eq:p=q} a 
{\sl strong center-local minimizer}
if there is a relative neighborhood ${\cal N}$ of 
$\bar{\boldsymbol c}$ such that
\[
F_{p, p}(\bar{\boldsymbol c},\bar{\boldsymbol z})
\leq F_{p, p}(\boldsymbol c,\boldsymbol z)
\]
for every $\boldsymbol{z}$ feasible to ~\eqref{eq:p=q} and every
$\boldsymbol c\in{\cal N}$.

\begin{theorem}\label{th:reduced p=q} \rm
Suppose $p=q$, let $(\bar{\boldsymbol c},\bar{\boldsymbol z})$ 
be feasible for~\eqref{eq:pq} and 
$\bar t_{sk}=\bar z_{sk}\lVert\xi^s-\bar c^k\rVert_p^p$.  The 
following statements are equivalent for the 
sets in~\eqref{eq:matched index sets}:

\noindent (A)
$(\bar{\boldsymbol c},\bar{\boldsymbol z},\bar{\boldsymbol t})$ 
is a global minimizer of~\eqref{eq:reduced MIP p=q};

\noindent (B)
$(\bar{\boldsymbol c},\bar{\boldsymbol z})$ is a strong 
center-local minimizer of~\eqref{eq:pq};

\noindent (C)
$(\bar{\boldsymbol c},\bar{\boldsymbol z})$ is a strong 
center-local minimizer of~\eqref{eq:p=q}.
\end{theorem}

\begin{proof} For simplicity, denote 
\[
{\cal Z} \, \triangleq \,
\left\{ \, \boldsymbol z \in \{0,1\}^{N\times K} \ \left| \
\displaystyle{
\sum_{k=1}^K
} \, z_{sk} = 1 \ \text{for every } s \in [ \, N \, ] \, 
\right. \right\}
\]
and
\[
\vartheta(\boldsymbol{c}) \, \triangleq \, \displaystyle{
\operatornamewithlimits{\mbox{\bf minimum}}_{\boldsymbol{z} 
\in \mathcal{Z}}
} \, F_{p, p}(\boldsymbol{c}, \boldsymbol{z}).
\]
Note that, for any $\widehat{\boldsymbol{c}} \in \mathcal{C}$ 
and $\widehat{\boldsymbol{z}} \in \mathcal{Z}$ such that 
$(\widehat{\boldsymbol{c}}, \widehat{\boldsymbol{z}})$ is feasible 
to \eqref{eq:pq} (with $p = q$), we have 
\[
\vartheta(\widehat{\boldsymbol{c}}) \, = \, F_{p,p}(\widehat{\boldsymbol{c}}, \widehat{\boldsymbol{z}}).
\]
We first establish the equivalence between (B) and (C). On
one hand, suppose (B) holds in a neighborhood $\mathcal{N}$ 
of $\bar{\boldsymbol c}$. For any ${\boldsymbol c} \in \mathcal{N}$,
choose $\boldsymbol{z}^{\boldsymbol{c}} \in \displaystyle{
\operatornamewithlimits{\mbox{\bf argmin}}_{z \in \mathcal{Z}} 
} \, F_p(\boldsymbol{c}, \boldsymbol{z})$. Then, 
$(\boldsymbol{c}, {\boldsymbol{z}}^{\boldsymbol{c}})$ is feasible 
to \eqref{eq:pq}. Together with the fact that 
$(\bar{\boldsymbol{c}}, \bar{\boldsymbol{z}})$ is a strong 
center-local minimizer of \eqref{eq:pq}, we have
\[
F_{p, p}(\bar{\boldsymbol{c}}, \bar{\boldsymbol{z}}) \leq F_{p, p}({\boldsymbol{c}}, {\boldsymbol{z}}^{\boldsymbol{c}}) = \vartheta(\boldsymbol{c}) \leq F_{p, p}(\boldsymbol{c}, \boldsymbol{z}) \quad  \forall \boldsymbol{z} \in \mathcal{Z}.
\]
This implies (C). On the other hand, the feasible set of \eqref{eq:pq} is a subset of the feasible set of \eqref{eq:p=q}, which implies that (C) $\Rightarrow$ (B). Therefore, (B) and (C) are equivalent. 

Next, we show that (A) and (B) are equivalent. The proof of 
(A) $\Rightarrow$ (B) is similar to the corresponding part of
Theorem~\ref{th:fixed-point localmin}: all admissible fixings 
remain valid in a neighborhood of $\bar{\boldsymbol c}$, and every 
nearby feasible point of \eqref{eq:pq} extends to a feasible point 
of~\eqref{eq:reduced MIP p=q}.  To prove (B) $\Rightarrow$ (A), let
\((\boldsymbol c,\boldsymbol z,\boldsymbol t)\) be an arbitrary 
feasible triplet of the matched-norm reduced problem 
\eqref{eq:reduced MIP p=q}, and denote its objective by
\[
\Psi(\boldsymbol c,\boldsymbol z,\boldsymbol t)
\triangleq
\frac{1}{N}
\sum_{(s,k)\in{\cal J}_c}t_{sk}
+
\frac{1}{N}
\sum_{(s,k)\in{\cal J}_1}
\lVert \xi^s-c^k\rVert_p^p.
\]
For every \((s,k)\in{\cal J}_c\), we know that
\[
t_{sk}\geq
z_{sk}\lVert \xi^s-c^k\rVert_p^p.
\]
Indeed, this follows from the big-\(M\) constraint when 
\(z_{sk}=1\)
and from \(t_{sk}\geq0\) when \(z_{sk}=0\). Together with 
the fixings on
\({\cal J}_1\) and \({\cal J}_0\), this gives
\begin{equation}\label{eq:reduced-dominates-F}
\Psi(\boldsymbol c,\boldsymbol z,\boldsymbol t) \, \geq \,
F_{p, p}(\boldsymbol c,\boldsymbol z).
\end{equation}
Moreover, we know that  
$\Psi(\bar{\boldsymbol c},\bar{\boldsymbol z},
\bar{\boldsymbol t}) \, = \,
F_{p, p}(\bar{\boldsymbol c},\bar{\boldsymbol z})$,
due to $\bar t_{sk}=\bar z_{sk}\lVert\xi^s-\bar c^k\rVert_p^p$. 
It therefore suffices to prove
$F_{p, p}(\boldsymbol c,\boldsymbol z) \geq
F_{p, p}(\bar{\boldsymbol c},\bar{\boldsymbol z})$.
Since
${\cal J}_0={\cal J}_>(\bar{\boldsymbol c})$, a key observation 
is that the center that can be selected by
$\boldsymbol z$  is the nearest one in $\bar{\boldsymbol c}$ 
(can be tied) for every sample.  This implies that
\begin{equation}\label{eq:matched-tie-value}
 F_{p, p}(\bar{\boldsymbol c},\boldsymbol z)
 =F_{p, p}(\bar{\boldsymbol c},\bar{\boldsymbol z}).
\end{equation}
More specifically, for each
\(s\in[N]\), define
\[
{\cal K}_s^\star
\triangleq
\operatornamewithlimits{\arg\min}_{k\in[K]}
\lVert \xi^s-\bar c^k\rVert_p,
\qquad
d_s^\star
\triangleq
\min_{k\in[K]}
\lVert \xi^s-\bar c^k\rVert_p.
\]
By the definition of \({\cal J}_{>}(\bar{\boldsymbol c})\), we have $z_{sk}=0, \; \; \forall \, k \notin {\cal K}_s^\star,$ and 
\[
(s,k)\notin{\cal J}_{>}(\bar{\boldsymbol c})
\; \iff \;
\lVert \xi^s-\bar c^k\rVert_p
\leq
\min_{k'\neq k}\lVert \xi^s-\bar c^{k'}\rVert_p 
\; \iff \;
k\in{\cal K}_s^\star.
\]
Therefore, combining with the feasibility of the reference assignment \(\bar{\boldsymbol z}\) to \eqref{eq:pq}, we have  
\[
\sum_{k=1}^K
z_{sk}\lVert \xi^s-\bar c^k\rVert_p^p
\; = \;
\sum_{k\in{\cal K}_s^\star}
z_{sk}\lVert \xi^s-\bar c^k\rVert_p^p \;
= \;
\min_{k\in[K]}
\lVert \xi^s-\bar c^k\rVert_p^p = \; 
\sum_{k\in{\cal K}_s^\star}
\bar z_{sk}\lVert \xi^s-\bar c^k\rVert_p^p, 
\]
which implies the equation \eqref{eq:matched-tie-value}. Now define, $\boldsymbol c^\lambda
\triangleq
(1-\lambda)\bar{\boldsymbol c}
+
\lambda\boldsymbol c$ for \(\lambda\in(0,1]\). 
Since \({\cal C}\) is convex,
\(\boldsymbol c^\lambda\in{\cal C}\). For every sufficiently small
\(\lambda>0\), \(\boldsymbol c^\lambda\) lies in the neighborhood
appearing in (B), and the strict inequalities defining
\({\cal J}_1\) and \({\cal J}_0\) remain valid at
\(\boldsymbol c^\lambda\).

Let \(\boldsymbol z^\lambda\) be any nearest-center assignment at
\(\boldsymbol c^\lambda\). The preserved strict inequalities imply that
\[
z^\lambda_{sk}=1
\quad\text{for }(s,k)\in{\cal J}_1,
\qquad
z^\lambda_{sk}=0
\quad\text{for }(s,k)\in{\cal J}_0.
\]
In particular, \((\boldsymbol c^\lambda,\boldsymbol z^\lambda)\) is
feasible for the original clustering problem. Set $t^{\lambda}_{sk} := z^{\lambda}_{sk}\|\xi^s - (c^{\lambda})^k\|_p^p$. By the strong
center-local optimality in (B),
\[
F_{p, p}(\bar{\boldsymbol c},\bar{\boldsymbol z}) = \Psi(\bar{\boldsymbol c},\bar{\boldsymbol z},
     \bar{\boldsymbol t})
\leq \Psi(\boldsymbol c^\lambda,\boldsymbol z^\lambda,
     \boldsymbol t^\lambda) = 
F_{p, p}(\boldsymbol c^\lambda,\boldsymbol z^\lambda).
\]
Moreover, since \(p=q\), assigning each sample to a nearest center minimizes the
\(\ell_p^p\) clustering cost for fixed centers. Hence
\[
F_{p, p}(\boldsymbol c^\lambda,\boldsymbol z^\lambda)
\leq
F_{p, p}(\boldsymbol c^\lambda,\boldsymbol z).
\]
Note that the function
\(F_{p, p}(\cdot,\boldsymbol z)\) is convex for fixed \(\boldsymbol z\). Therefore,
\[
F_{p, p}(\boldsymbol c^\lambda,\boldsymbol z)
\leq
(1-\lambda)F_{p, p}(\bar{\boldsymbol c},\boldsymbol z)
+
\lambda F_{p, p}(\boldsymbol c,\boldsymbol z).
\]
Combining the preceding inequalities with
\eqref{eq:matched-tie-value} gives
\[
F_{p, p}(\bar{\boldsymbol c},\bar{\boldsymbol z})
\leq
(1-\lambda)
F_{p, p}(\bar{\boldsymbol c},\bar{\boldsymbol z})
+
\lambda F_{p, p}(\boldsymbol c,\boldsymbol z).
\]
Since \(\lambda>0\), this implies
\[
F_{p, p}(\boldsymbol c,\boldsymbol z)
\geq
F_{p, p}(\bar{\boldsymbol c},\bar{\boldsymbol z}).
\]
Since the reduced feasible triplet
\((\boldsymbol c,\boldsymbol z,\boldsymbol t)\) was arbitrary, this proves (A).
\end{proof}

\begin{remark} \rm
Central to the above proof, the 
equality~\eqref{eq:matched-tie-value} is ensured
by the choice  ${\cal J}_0={\cal J}_>(\bar{\boldsymbol c})$ and $p=q$, 
without which a free binary variable could
select a center that is strictly non-nearest at the reference 
point, therefore jeopardizing \eqref{eq:matched-tie-value} and 
the reverse implication (B) $\Rightarrow$ (A).  \hfill $\Box$
\end{remark}

%
% In Theorem \ref{th:reduced p=q}, we do not require 
% Assumption \ref{ass: sample-center-unique} due to $p = q$ and the particular 
% choice of ${\cal J}_0={\cal J}_>(\bar{\boldsymbol c})$. This choice is important
% for the reverse implication in Theorem~\ref{th:reduced p=q}.  With only
% ${\cal J}_0\subseteq{\cal J}_>(\bar{\boldsymbol c})$, a free binary variable could
% select a center that is strictly non-nearest at the reference point, and
% identity~\eqref{eq:matched-tie-value} would fail. Of course, 
% if Assumption \ref{ass: sample-center-unique} holds, the statements (A) and (B) 
% in Theorem \ref{th:reduced p=q} are also equivalent to 
% that $(\boldsymbol{\bar{c}}, \boldsymbol{\bar{z}})$ is a local minimizer 
% of \eqref{eq:pq}.
% \end{remark}

% Until now, we have characterized the strong center-local 
% minimizer of the mixed-norm clustering 
% model \eqref{eq:pq} by its corresponding restricted subproblems. 
% Next, we will focus on an efficient implementation of the PIP 
% method for $q = 1$, for which the nonconvex dc-constraints 
% admit a tractable convex inner approximation.

\section{The PIP Method for Solving \eqref{eq:pq} with 
\texorpdfstring{$q=1$}{q=1}}
% Convexification and Implementation}
\label{sec: PIP-q=1}

Theorem~\ref{th:fixed-point localmin} suggests the following iterative
scheme to compute a locally optimal solution of (\ref{eq:pq}) for general
pairs $(p,q)$.   
Given a feasible pair $\boldsymbol{\bar{y}} \triangleq 
( \boldsymbol{\bar{c}},\boldsymbol{\bar{z}} )$, select a pair of disjoint
index sets ${\cal J}_1$ and ${\cal J}_0$ satisfying 
(\ref{eq:index set inclusions}), then solve the subproblem (\ref{eq:reduced MIP}).
Repeat these steps until the given pair 
$( \boldsymbol{\bar{c}},\boldsymbol{\bar{z}} )$ is an optimal solution
of its own defined problem (\ref{eq:reduced MIP}). % We will subsequently 
Postponing the details of 
an efficient implementation to select the index sets and to terminate the iterations,
we first discuss a practical way to solve (\ref{eq:reduced MIP}), whose challenge
lies in the difference-of-convex (dc) constraint 
$\lVert\xi^s-c^k\rVert_q\leq\lVert\xi^s-c^{k'}\rVert_q$.  
We will develop a convex inner approximation to handle this constraint.

\subsection{A convex inner approximation}
\label{subsec:q1-workhorse}

Corresponding to a given $\bar{\boldsymbol{c}}$, 
let ${\cal J}_1$ and ${\cal J}_0$ be any admissible fixing sets 
satisfying
\eqref{eq:index set inclusions}.  If $(s,k)\in{\cal J}_1$, then the assignment
equation must fix $z_{sk'}=0$ for every $k'\ne k$.  To eliminate these redundant
variables explicitly, define
\[ \begin{array}{lll}
{\cal J}_0^{\,\prime} & \triangleq & {\cal J}_0 \, \cup \,
\left\{ \, (s,k^{\prime} ) \, \mid \, \text{there exists } 
k \neq k^{\prime}
\text{ with }(s,k)\in{\cal J}_1 \, \right\} \\ [0.1in]
{\cal J}_c^{\,\prime} & \triangleq & [\, N \, ] \, \times \, [K] \, 
\setminus \, \left( \, {\cal J}_1 \, \cup \, {\cal J}_0^{\,\prime}
\, \right) \\ [0.1in]
{\cal J}_s^{\,\prime} & \triangleq &
\left\{ \, s \in [ \, N \, ] \, \mid \, \text{there exists } k 
\text{ with } (s,k) \in {\cal J}_c^{\,\prime} \, \right\}.
\end{array} \]
Thus ${\cal J}_0^{\,\prime}$ contains both the assignments 
fixed directly to zero and those forced to zero by a one-fixing.  
Binary variables remain only for pairs in ${\cal J}_c^{\,\prime}$, 
and ${\cal J}_s^{\,\prime}$ identifies
the samples whose assignments have not yet been determined.
%
% \noindent In view of the last implication, we expand the index set 
% ${\cal J}_0$ to include those pairs $(s,k^{\prime})$ for which there exists
% $k \neq k^{\prime}$ such that $(s,k) \in {\cal J}_1$ and denote the 
% resulting expanded set by ${\cal J}_0^{\, \prime}$.  We let
% \[
% {\cal J}_c^{\, \prime} \, \triangleq \, [ \, N \, ] \times \, [ \, K \, ] 
% \, \setminus \,
% \left( \, {\cal J}_1 \, \cup \, {\cal J}_0^{\, \prime} \, \right),
% \]
% which contains all pairs $(s,k)$ whose binary variables $z_{sk}$ 
% have yet to be determined, and thus need to be solved for. 
% Associated with ${\cal J}_c^{\, \prime}$, define the subset 
% of samples 
% \[
% {\cal J}_s^{\, \prime} \, \triangleq \, \left\{ \, s \, \in \, 
% [ \, N \, ] \, \mid \, \exists \, k \, \in \, [ K ]
% \mbox{ such that } (s,k) \in {\cal J}_c^{\, \prime} \, \right\}.
% \]
Based on these revised index sets, we can further reduce problem
(\ref{eq:reduced MIP}) as follows:
\begin{equation} \label{eq:reduced MIP exposed} 
\begin{array}{ll}
\underset{\boldsymbol{c} \, \in \, {\cal C}; \, \boldsymbol{z}; 
\, \boldsymbol{t}}{\mbox{\bf minimize}} & \displaystyle{
\frac{1}{N}
} \, \displaystyle{
\sum_{(s,k) \in {\cal J}_c^{\, \prime}}
} \, t_{sk} + \displaystyle{
\frac{1}{N}
} \, \displaystyle{
\sum_{(s,k) \in {\cal J}_1}
} \ \| \, \xi^s - c^k \, \|_p^p \\ [0.2in]
\mbox{\bf subject to} & \displaystyle{
\sum_{k=1}^K                                          
} \, z_{sk} = 1 \, \left( \, = \, \displaystyle{
\sum_{k\, \mid \, (s,k) \in {\cal J}_c^{\, \prime}}
} \, z_{sk} \, \right) \hspace{0.25in} \forall \, 
s \, \in \, {\cal J}_s^{\, \prime} \\ [0.25in]
& \left\{ \begin{array}{l}
\mbox{for all $(s,k) \, \in \, {\cal J}_c^{\, \prime}$:} \\ [0.1in]
t_{sk} \, \geq \, \| \, \xi^s - c^k \, \|_p^p - M(1 - z_{sk});
\epc t_{sk} \, \geq \, 0 \\ [0.1in]
\| \, \xi^s - c^k \, \|_q - \| \, \xi^s - c^{k^{\prime}} \, \|_q \, 
\leq \, M \, (1 - z_{sk}), \quad \forall \, k^{\prime} \, \neq \, k 
\\ [0.1in]
z_{sk} \, \in \, \{ \, 0,1 \, \}
\end{array} \right\} \\ [0.5in]
& \| \, \xi^s - c^k \, \|_q \, \leq \,
\| \, \xi^s - c^{k^{\prime}} \, \|_q \epc \forall \, 
(s,k) \, \in \, {\cal J}_1 \\ [0.1in]
& z_{sk} \, = \, 1, \ (s,k) \, \in \, {\cal J}_1; 
\ \mbox{ \bf and } \
z_{sk} \, = \, 0, \ (s,k) \, \in \, {\cal J}_0^{\, \prime}.	
\end{array}
\end{equation}
Next, we solve the nonconvex
problem (\ref{eq:reduced MIP exposed}) (or equivalently
\eqref{eq:reduced MIP}) with dc constraints by linearizing the concave term in each difference-of-convex function, which is
$-\| \, \xi^s - c^{k^{\prime}} \, \|_q$ in (\ref{eq:reduced MIP exposed}).
For $q = 1$, this is done by lower bounding 
\begin{equation}
\label{eq:norm lower bound}
\| \, \xi^s - c^{k^{\prime}} \|_1 \, = \, \displaystyle{
\sum_{j=1}^d
} \, | \, \xi^s_j - c^{k^{\prime}}_j \, | \, \geq \, \displaystyle{
\sum_{j=1}^d
} \, \bar{\sigma}_{s{k^{\prime}}j} \, ( \, \xi^s_j - c^{k^{\prime}}_j \, ),
\end{equation}
where 
\[\bar{\sigma}_{skj} \triangleq \left\{ \begin{array}{ll}
\mbox{sgn}(\xi^s_j - \bar{c}^k_j) & \mbox{if $\xi^s_j \neq \bar{c}^k_j$}
\\ [5pt]
0 & \mbox{otherwise.} 
\end{array} \right.\]
The problem (\ref{eq:reduced MIP exposed}) with
$q = 1$ is approximated by the following IP whose constraints are 
all convex except for the binary variables:
\begin{equation}\label{eq:reduced MIP final}
\begin{array}{cll}
\displaystyle{ 
\operatornamewithlimits{\mbox{\bf minimize}}_{\boldsymbol c
\in {\cal C},\,\{ z_{sk}, t_{sk}:(s,k)\in{\cal J}_c^{\,\prime} \}}
} & \displaystyle{
\frac{1}{N}
} \, \left[ \, \displaystyle{
\sum_{(s,k)\in{\cal J}_c^{\,\prime}}
} \, t_{sk} + \displaystyle{
\sum_{(s,k)\in{\cal J}_1}
} \, \lVert \, \xi^s-c^k \, \rVert_p^p \, \right] & \\ [0.2in]
\mbox{\bf subject to} & \displaystyle{
\sum_{k : (s,k)\in{\cal J}_c^{\,\prime}}
} \, z_{sk} \, = \, 1, & s \in {\cal J}_s^{\,\prime} \\ [0.2in]
& t_{sk} \, \geq \, \lVert \, \xi^s-c^k \, \rVert_p^p -
M ( 1-z_{sk} ), \quad t_{sk} \, \geq \, 0, 
& (s,k) \in {\cal J}_c^{\,\prime} \\ [0.1in]
& \lVert \, \xi^s-c^k \, \rVert_1 - \displaystyle{
\sum_{j=1}^d
} \, \bar\sigma_{sk^{\prime}j}(\xi_j^s-c_j^{k^{\prime}})
\, \leq \, M ( 1-z_{sk} ), & (s,k) \in {\cal J}_c^{\,\prime},
\ k^{\prime} \neq k \\ [0.1in]
& z_{sk} \, \in \, \{0,1\}, & (s,k) \in {\cal J}_c^{\,\prime}
\\ [0.1in]
\mbox{\bf and} & \lVert \, \xi^s-c^k \, \rVert_1 \, \leq \, 
\displaystyle{
\sum_{j=1}^d
} \, \bar\sigma_{sk^{\prime}j}(\xi_j^s-c_j^{k^{\prime}}),
& (s,k) \in {\cal J}_1, \ k^{\prime} \neq k.
\end{array}
\end{equation}
Problem (\ref{eq:reduced MIP final}) is a further restriction of the
clustering $\ell_{p,1}$ problem \eqref{eq:pq} because its constraint
set is a subset
of the constraint set of (\ref{eq:reduced MIP}).  We state a direct
connection between problems (\ref{eq:reduced MIP final}) and 
(\ref{eq:pq}) for $q = 1$ under the following coordinate-wise 
center-sample separation assumption.

% \begin{assumption} \rm
% \label{ass: coordinate-cen-sam-sep}
% \end{assumption}

\begin{theorem} \label{th:fixed-point localmin linearized} \rm
Let 
$( \boldsymbol{\bar{c}},\boldsymbol{\bar{z}})$ be a feasible solution to (\ref{eq:pq}), and set $\bar{t}_{sk} = \bar{z}_{sk}\|\xi^s - \bar{\boldsymbol{c}}^k\|_p^p, \forall (s, k) \in [N] \times [K]$.  Suppose that
$\bar c_j^k \neq \xi_j^s$ for 
every triplet $(s,k,j)$ in $[ N ] \times [ K ] \times [ d ]$.  
Among the following statements:

\noindent (A) 
$(\boldsymbol{\bar{c}},\boldsymbol{\bar{z}},\boldsymbol{\bar{t}} )$ 
is (globally) optimal for (\ref{eq:reduced MIP final}) for 
some pair of disjoint index sets ${\cal J}_1$ and ${\cal J}_0$ 
satisfying the inclusions (\ref{eq:index set inclusions});

\noindent (B) the pair 
$( \boldsymbol{\bar{c}},\boldsymbol{\bar{z}})$ is a strong 
center-local minimizer of \eqref{eq:pq} for some neighborhood 
$\mathcal{N}$ of $\boldsymbol{\bar{c}}$; 

\noindent (C) $( \boldsymbol{\bar{c}},\boldsymbol{\bar{z}} )$ 
is a locally optimal solution of (\ref{eq:pq});

\gap 

\noindent it holds that (A) $\Rightarrow$ (B) $\Rightarrow$ (C).  
If Assumption \ref{ass: sample-center-unique} holds, 
then (C) $\Rightarrow$ (A).
\end{theorem}

\begin{proof} Let the neighborhood ${\cal N}$ be 
chosen such that for all $\boldsymbol{c} \in {\cal N}$ for all
$s \in [ N ]$:
(i) the implications in (\ref{eq:continuity}) hold for all
$(k,k^{\prime}) \in [ K ] \times [ K ]$, and (ii) $c^k_j \neq \xi^s_j$
for all $(k, j) \in [ K ] \times [ d ]$. 
Let $( \boldsymbol{c},\boldsymbol{z} ) \in {\cal N}
\times \in \{ \, 0,1 \, \}^{NK}$ be an arbitrary feasible pair to
(\ref{eq:pq}). By (ii), it follows that restricted to this
neighborhood ${\cal N}$ of $\boldsymbol{c}$, 
problem (\ref{eq:reduced MIP final})  
is exactly the same as problem (\ref{eq:reduced MIP}).  Therefore,
Theorem~\ref{th:fixed-point localmin} applies, and the statements 
follow directly.
\end{proof}

\subsection{A summary of formulations and connections}

We summarize the logical chains connecting the four related
formulations for the $\ell_{p,q}$ mixed-norm clustering problem.

% \noindent
% $(\mathbf{P}_{\mathrm{par}})$ is solved via the following MIP, obtained by
% introducing auxiliary variables $t_{sk}\geq 0$:
% \begin{equation}\label{eq:Ppar-MIP}
% (\mathbf{P}_{\mathrm{par}}\text{-MIP}) \quad
%    \begin{aligned}
%        &\mbox{\bf minimize}
%          &&\frac{1}{N}\!\sum_{(s,k)\in\mathcal{J}^0_\tau}\!t_{sk}
%            +\frac{1}{N}\!\sum_{(s,k)\in\mathcal{J}^+_\tau}
%            \!\|c^k-\xi^s\|_2^2 \\[6pt]
%        &\mbox{\bf subject to}
%           &&t_{sk}\geq\|c^k-\xi^s\|_2^2-M(1-z_{sk}),
%            \quad\forall\,(s,k)\in\mathcal{J}^0_\tau, \\[6pt]
%        &&&t_{sk}\geq 0,\quad\forall\,(s,k)\in\mathcal{J}^0_\tau, \\[6pt]
%        &&&\sum_{k:\,(s,k)\in\mathcal{J}^0_\tau}\!z_{sk}=1,
%            \;\forall\,s\text{ s.t. }(s,k)\in\mathcal{J}^0_\tau
%            \text{ for some }k, \\[6pt]
%        &&&\sum_{j=1}^{d}|c^k_j-\xi^s_j|
%            -\sum_{j=1}^{d}|c^{k'}_j-\xi^s_j|
%            \leq M(1-z_{sk}),
%            \quad\forall\,(s,k)\in\mathcal{J}^0_\tau,\;k'\neq k, \\[6pt]
%        &&&z_{sk}\in\{0,1\},\quad\forall\,(s,k)\in\mathcal{J}^0_\tau.
%    \end{aligned}
%\end{equation}

% \subsection{Convergence Analysis}
% \label{subsec:convergence}
%
% Figure~\ref{fig:chain} displays the logical chain connecting all five
% formulations introduced in this section.

\begin{figure}[H]
\centering
\begin{tikzpicture}[
  box/.style = {draw, rounded corners=3pt,
                minimum width=2.2cm, minimum height=0.82cm,
                align=center, font=\small},
  lbl/.style  = {font=\scriptsize, align=center},
  >=stealth, thick
]

%% NODES
\node[box] (eq3) at (0.0, 0.0) {$(\ref{eq:pq})$\\[1pt]original};
% $(\mathbf{P}_0)$};
\node[box] (eq4)  at (4.0, 0.0) {$(\ref{eq:MIP})$\\[1pt]full MIP};
\node[box] (Prst) at (4.0,-2.5) {$(\ref{eq:reduced MIP}) \Leftrightarrow 
(\ref{eq:reduced MIP exposed})$\\[1pt]partial MIP};
\node[box] (Ppar) at (9.2, -2.5) {$(\ref{eq:reduced MIP final})$ for $q = 1$
\\[1pt]computational workhorse};
% \node[box, minimum width=2.6cm] (PM) at (10.5,-2.5)
%    {$(10)$\\[1pt]$(\mathbf{P}_{\mathrm{par}}\text{-MIP})$};

%% EDGES
\draw[<->] (eq3) --
    node[lbl, above]{globally \\ equivalent}
    node[lbl, below]{via big-M} (eq4);

\draw[->] (eq4.south) --
    node[lbl, left, xshift=-3pt]{restriction \\ for PIP}
    (Prst.north);

\draw[->] (Prst) --
    node[lbl, above]{convexification} (Ppar);

% \draw[<-] (Ppar) --
%    node[lbl, above]{working formula\\(big-$M$)} (PM);

%% CONVERGENCE BRACKET 
% \draw[thin,gray] (Prst.south) -- (4.0, -3.3);
% \draw[thin,gray] (Ppar.south) -- (7.2, -3.3);
% \draw[thin,gray] (PM.south)   -- (10.5,-3.3);
% \draw[thin,gray] (4.0,-3.3)   -- (10.5,-3.3);
% \draw[->,thin,gray] (7.25,-3.3) -- (7.25,-3.52);
% \node[lbl] at (7.25,-3.72)
%    {global opt.\ $\Rightarrow$ local opt.\ for $(3)$ and $(4)$};

\end{tikzpicture}
% \caption{\label{fig:chain}}
\end{figure}

\begin{center}
\begin{tabular}{ccccc}
self-global of (\ref{eq:reduced MIP}) & 
$\xRightarrow{\rm{Thm.~\ref{th:fixed-point localmin}}}$ & 
strong local of (\ref{eq:pq}) & 
$\xRightarrow[\text{unique center}]{\text{sample-wise }}$ &
self-global of (\ref{eq:reduced MIP}) \\ [0.1in]
$\Big\Downarrow$ \\ [0.1in]
self-global of (\ref{eq:reduced MIP final}) &
$\xRightarrow[q = 1]{\rm Thm.~\ref{th:fixed-point localmin linearized}}$ & 
strong local of (\ref{eq:pq}) &
$\xRightarrow[\text{+ separation}]{\text{unique center}}$ &
self-global of (\ref{eq:reduced MIP final}).
\end{tabular}
\end{center}

\subsection{The construction of the fixing index sets}
\label{subsec:temperature-fixing}

While (\ref{eq:reduced MIP final}) is the computational workhorse 
of the PIP method,
its implementation requires some specifications to enhance its practical
performance.  First is the choice of the index sets to fix the binary 
variables. Let $(\bar{\boldsymbol c},\bar{\boldsymbol z})$ be the current feasible
solution at hand, and let $\rho_{sk}(\tau)$ be the softmax probabilities
in~\eqref{eq:softmax} for (adjustable) $\tau > 0$. Together with a given
$\varepsilon \in ( 0, 1/K )$, we set
\[ \begin{array}{lll}
{\cal J}_1(\tau) & \triangleq &
\left\{ \, \{(s,k) \, \mid \, \rho_{sk}(\tau) \, \geq \, \thalf \, \right\}, \quad \left(\mbox{replace ``$\geq$'' by ``$>$'' if $K = 2$}\right)\\ [0.1in]
{\cal J}_0^{\, \prime}(\tau) & \triangleq &
\underbrace{\left\{ \, (s,k) \, \mid \, \rho_{sk}(\tau) \, < \, \displaystyle{
\frac{1}{K}
} - \varepsilon \, \right\}}_{\mbox{$= {\cal J}_0(\tau)$}} \, \bigcup \, 
\underbrace{\left\{ \, (s,k^{\prime}) \, \mid \, \exists \, k \neq k^{\prime} 
\mbox{ such that } (s,k) \, \in \, {\cal J}_1(\tau) \, \right\}}_{\mbox{
implies $\bar{z}_{sk^{\prime}} = 0$}},\\ [0.4in]
{\cal J}_c^{\, \prime}(\tau) & \triangleq & [ \, N \, ] \times \, [ \, K \, ] 
\, \setminus \,
\left( \, {\cal J}_1(\tau)\, \cup \, {\cal J}_0^{\, \prime}(\tau) \, \right), 
\\ [0.1in]
% & = & \underbrace{\left\{ \, (s,k) \, \mid \, \displaystyle{
% \frac{1}{K}
% } - \varepsilon \, \leq \, \rho_{sk}(\tau) \, < \, \thalf \, \right\}}_{\mbox{
% $= {\cal J}_c(\tau)$}} \, \bigcap \, 
% \underbrace{\left\{ \, (s,k^{\prime}) \, \mid \, \rho_{sk}(\tau) 
% \, < \, \thalf \, \mbox{ for all $k \neq k^{\prime}$} \,\right\}}_{\mbox{
% complement of ${\cal J}_0(\tau)$ in ${\cal J}_0^{\prime}(\tau)$}} 
% \\ [0.4in]
{\cal J}_s^{\prime}(\tau) & \triangleq & \left\{ \, s \, \in \, [ \, N \, ] 
\, \mid \, \exists \, k \, \in \, [ K ]
\mbox{ such that } (s,k) \in {\cal J}_c^{\, \prime}(\tau) \, \right\}.
\end{array} \]
These sets play the roles of ${\cal J}_1$, ${\cal J}_0^{\, \prime}$,
${\cal J}_c^{\, \prime}$, and ${\cal J}_s^{\,\prime}$ in 
(\ref{eq:reduced MIP final}), respectively.
The choice of $\tau$ is based on the principle that we want to keep 
the number of unfixed binary variables (which are indexed by the
elements of the set ${\cal J}_c^{\, \prime}$) not large (while maintaining
the feasibility of the sum constraint )
and yet not small,
the reason for the former is to control the computational effort and 
the reason for the latter is not to be too aggressive in fixing the 
undetermined binary variables. Since
$\displaystyle{
\lim_{\tau \uparrow \infty}
} \, \rho_{sk}(\tau) = 1/K$ for all 
$(s, k) \in [ \, N \, ] \times \, [ \, K \, ]$, 
it follows that for $K > 2$ and sufficiently large $\tau$, $
 {\cal J}_1(\tau)={\cal J}_0(\tau)=\emptyset,
 \,\,
 {\cal J}_c^{\,\prime}(\tau)=[N]\times[K].$
In our implementation, to balance the cardinality of the unfixed binary 
variable set, we
increase $\tau$ till a prescribed upper bound is reached if the iteration 
does not progress  well, or  readjust $\tau$ to its initial value if there 
is an improvement in the objective.  

\subsection{A working set implementation for fixed assignments}
\label{subsec:working set}

The other control in the solution of (\ref{eq:reduced MIP final}) 
is the number of the constraints:
\begin{equation} \label{eq:linearized dc}
\| \, \xi^s - c^k \, \|_1 \, \leq \, \displaystyle{
\sum_{j=1}^d
} \, \bar{\sigma}_{sk'j} \, ( \, \xi^s_j - c^{k^{\prime}}_j \, )
\epc \forall \, (s,k) \, \in \, {\cal J}_1(\tau), \, k'\neq k
\end{equation}
which although are linearized, may create a large model if they are all
included. Indeed, not all these constraints are needed because their role is
to reinforce the validity of fixing $z_{sk} = 1$ for 
$(s,k) \in {\cal J}_1(\tau)$.  This consideration leads to the 
maintenance of a working set ${\cal J}_w(\tau)\subseteq 
{\cal J}_1(\tau)$ and a remove-and-insert strategy that is
embedded in an inner loop for solving (\ref{eq:reduced MIP final}) 
as a way 
of solving (\ref{eq:reduced MIP}).  {For
$(s,k)\in{\cal J}_w(\tau)$, we include the associated linearized constraints 
in~\eqref{eq:linearized dc} associated with such a pair $(s,k)$ and 
$k^{\prime} \neq k$ in the subproblem.} For fixed $\tau$, initialize 
${\cal J}_w(\tau)=\emptyset$ and solve the subproblem which
we label as ${\cal P}$:
\begin{equation} \label{eq:reduced MIP final removed}
\begin{array}{l}
\underset{\boldsymbol{c} \, \in \, {\cal C}; \, \boldsymbol{z}; 
\, \boldsymbol{t}}{\mbox{\bf minimize}} \ \displaystyle{
\frac{1}{N}
} \, \displaystyle{
\sum_{(s,k) \in {\cal J}_c^{\, \prime}(\tau)}
} \, t_{sk} + \displaystyle{
\frac{1}{N}
} \, \displaystyle{
\sum_{(s,k) \in {\cal J}_1(\tau)}
} \ \| \, \xi^s - c^k \, \|_p^p  \\ [0.2in]
\mbox{\bf subject to} \ \displaystyle{
\sum_{k=1}^K
} \, z_{sk} = 1 \, \left( \, = \, \displaystyle{
\sum_{k\, \mid \, (s,k) \in {\cal J}_c^{\, \prime}(\tau)}
} \, z_{sk} \, \right) \hspace{0.25in} \forall \, 
s \, \in \, {\cal J}_s^{\, \prime}(\tau) \\ [0.2in]
\left\{ \begin{array}{l}
\mbox{for all $(s,k) \, \in \, {\cal J}_c^{\, \prime}(\tau)$:} \\ [0.1in]
t_{sk} \, \geq \, \| \, \xi^s - c^k \, \|_p^p - M(1 - z_{sk});
\epc t_{sk} \, \geq \, 0 \\ [0.1in]
\| \, \xi^s - c^k \, \|_1 - \displaystyle{
\sum_{j=1}^d
} \, \bar{\sigma}_{sk'j} \, ( \, \xi^s_j - c^{k^{\prime}}_j \, ) \, 
\leq \, M \, (1 - z_{sk}), \quad \forall \, k^{\prime} \, \neq \, k 
\\ [0.2in]
z_{sk} \, \in \, \{ \, 0,1 \, \}
\end{array} \right\} \\ [0.55in]
\| \, \xi^s - c^k \, \|_1 \, \leq \, \displaystyle{
\sum_{j=1}^d
} \, \bar{\sigma}_{sk'j} \, ( \, \xi^s_j - c^{k^{\prime}}_j \, )
\epc \forall \, (s,k) \, \in \, {\cal J}_w(\tau) \
\ \mbox{(a subset of ${\cal J}_1(\tau)$)}, \ k'\neq k \\ [0.2in]
z_{sk} \, = \, 1, \ (s,k) \, \in \, {\cal J}_1(\tau); 
\ \mbox{ \bf and } \
z_{sk} \, = \, 0, \ (s,k) \, \in \, {\cal J}_0^{\, \prime}(\tau).
\end{array}
\end{equation}
The solution to the above problem gives an iterate $\boldsymbol{\wh{c}}$. We then determine a 
corresponding $\boldsymbol{\wh{z}}$ by solving (\ref{eq:pq}) 
with $\boldsymbol{c}$ fixed
at $\boldsymbol{\wh{c}}$. Define
\[
{\cal K}_s(\widehat{\boldsymbol c}) \, \triangleq \, \displaystyle{
\operatornamewithlimits{\mbox{\bf argmin}}_{k\in[K]}
} \, \lVert \, \xi^s-\widehat c^{\, k} \, \rVert_1
\]
and choose $\widehat z_{sk}=1$ for one index
$
k \, \in \, \displaystyle{
\operatornamewithlimits{\mbox{\bf argmin}}_{k^{\prime} \in
{\cal K}_s(\widehat{\boldsymbol c})}
} \, \lVert \, \xi^s-\widehat c^{\, k^{\prime}} \, \rVert_p^p,$
and set $\widehat z_{sk^{\prime}}=0$ for all $k^{\prime} \neq k$. 
The obtained pair 
$(\boldsymbol{\wh{c}},\boldsymbol{\wh{z}})$ is feasible 
to (\ref{eq:pq}).  If
\begin{equation}
\label{eq:sufficient-decrease}
F_{p, 1}(\boldsymbol{\wh{c}},\boldsymbol{\wh{z}}) < F_{p, 1}(\boldsymbol{\bar{c}}, \boldsymbol{\bar{z}}),
\end{equation}
the candidate $(\boldsymbol{\wh{c}},\boldsymbol{\wh{z}})$ is 
accepted and exit the current PIP subproblem 
(\ref{eq:reduced MIP exposed}).  Otherwise, insert the violated
constraints into the current ${\cal J}_w(\tau)$. Denote the set 
of invalid fixed assignments
\begin{equation}\label{eq:working-violations}
{\cal V}(\widehat{\boldsymbol c};\tau) \, \triangleq \,
\left\{ \, (s,k)\in{\cal J}_1(\tau) \, \mid \,
\lVert \, \xi^s-\widehat c^k \, \rVert_1 \, > \, 
\min_{k^{\prime} \ne k} \lVert \, \xi^s-\widehat c^{k^{\prime}}
\, \rVert_1 \, \right\}.
\end{equation}
If this set is nonempty, update
\begin{equation}\label{eq:working-update}
 {\cal J}_w(\tau)
 \leftarrow{\cal J}_w(\tau) \, \cup \,
 {\cal V}(\widehat{\boldsymbol c};\tau)
\end{equation}
and resolve~\eqref{eq:reduced MIP final removed} without changing $\tau$.
If the set is empty and the sufficient decrease is not satisfied, the working set
loop terminates, and proceed to the outer procedure with an increasing value $\tau$. 
Note that, with the 
working set, there is no need to solve the full
inner approximation unless eventually
${\cal J}_w(\tau)={\cal J}_1(\tau)$. The next proposition 
gives the finite termination of the working set update.

\begin{proposition}[Finite termination of the working set loop] \rm
\label{prop:working-finite}
Fix $\tau>0$ and the incumbent $\bar{\boldsymbol c}$.  
Suppose that every instance
of~\eqref{eq:reduced MIP final removed} is solved to global 
optimality.  Then
the working set loop either accepts a point satisfying
\eqref{eq:sufficient-decrease} or terminates after at most
$|{\cal J}_1(\tau)|+1$ solves.
\end{proposition}

\begin{proof}
If $(s,k)\in{\cal J}_w(\tau)$, then~\eqref{eq:linearized dc} and
\eqref{eq:norm lower bound} imply
\[
 \lVert\xi^s-\widehat c^k\rVert_1
 \leq
 \sum_{j=1}^d\bar\sigma_{sk'j}
      (\xi_j^s-\widehat c_j^{k'})
 \leq\lVert\xi^s-\widehat c^{k'}\rVert_1
 \quad\text{for every }k'\ne k.
\]
Hence an index already in ${\cal J}_w(\tau)$ cannot belong to
${\cal V}(\widehat{\boldsymbol c};\tau)$.  Every unsuccessful 
iteration with a nonempty violation set therefore adds at least 
one new member of the finite
set ${\cal J}_1(\tau)$.  After at most $|{\cal J}_1(\tau)|$ 
iterations,
the next solve must lead either to acceptance or to an empty 
violation set.
\end{proof}

% \subsection{Details of PIP algorithm for the $\ell_{p, 1}$
% mixed-norm clustering model}
%
We now assemble the above discussion into the complete Algorithm~1
and make some remarks about its implementation.

\begin{algorithm}[t]
\caption{working set (inner) loop at a fixed $\tau$}
\label{alg:innerloop}
\begin{algorithmic}[1]
\Require Incumbent $(\bar{\boldsymbol c},\bar{\boldsymbol z})$,
         $\tau>0$
\Ensure  A pair $(\boldsymbol{\wh c},\boldsymbol{\wh z})$ and a status
         $\in\{\textsc{descent},\textsc{stable}\}$
\State $\bigl({\cal J}_1(\tau),{\cal J}_0^{\,\prime}(\tau),
       {\cal J}_c^{\,\prime}(\tau),{\cal J}_s^{\,\prime}(\tau)\bigr)
       \leftarrow\textsc{Partition}(\bar{\boldsymbol c},\tau)$
       \Comment{Section~\ref{subsec:temperature-fixing}}
\State $\bar\sigma_{sk^{\prime}j}\leftarrow
       \operatorname{sign}(\xi^s_j-\bar c^{\,k^{\prime}}_j)$
\State build ${\cal P}$, the instance
       of~\eqref{eq:reduced MIP final removed} determined by these index sets
       and $\bar\sigma$;\; ${\cal J}_w(\tau)\leftarrow\emptyset$
\Loop
       \Comment{terminates by Proposition~\ref{prop:working-finite}}
  \State $(\boldsymbol{\wh c},\boldsymbol{\wh z}^{\,w})
          \leftarrow\textsc{Solve}({\cal P})$
         \Comment{warm started}
  \State $\boldsymbol{\wh z}\leftarrow
          \ell_1\text{-}\textsc{Assign}(\boldsymbol{\wh c})$
  \If{$F_{p,1}(\boldsymbol{\wh c},\boldsymbol{\wh z})
       <F_{p,1}(\bar{\boldsymbol c},\bar{\boldsymbol z})$}
      \Comment{\eqref{eq:sufficient-decrease}}
    \State \Return $(\boldsymbol{\wh c},\boldsymbol{\wh z})$,
           \textsc{descent}
  \EndIf
  \State ${\cal V}(\boldsymbol{\wh c};\tau)\leftarrow
          \bigl\{(s,k)\in{\cal J}_1(\tau):
          \lVert\xi^s-\wh c^{\,k}\rVert_1>
          \min_{k^{\prime}\ne k}
          \lVert\xi^s-\wh c^{\,k^{\prime}}\rVert_1\bigr\}$
         \Comment{cf.~\eqref{eq:working-violations}}
  \If{${\cal V}(\boldsymbol{\wh c};\tau)=\emptyset$}
    \State \Return $(\bar{\boldsymbol c},\bar{\boldsymbol z})$,
           \textsc{stable}
  \EndIf
  \State append to ${\cal P}$ the constraints~\eqref{eq:linearized dc} indexed
         by ${\cal V}(\boldsymbol{\wh c};\tau)\setminus{\cal J}_w(\tau)$;\;
         ${\cal J}_w(\tau)\leftarrow{\cal J}_w(\tau)\cup
          {\cal V}(\boldsymbol{\wh c};\tau)$
         \Comment{row generation, cf.~\eqref{eq:working-update}}
\EndLoop
\end{algorithmic}
\end{algorithm}

\begin{algorithm}[t]
\caption{A progressive integer programming algorithm for the $\ell_{p,1}$
mixed-norm clustering model}
\label{alg:pip}
\begin{algorithmic}[1]
\Require Data $\Xi=\{\xi^s\}_{s=1}^N$, clusters $K$, warm start
         $\bar{\boldsymbol c}^{\,(0)}$, schedule
         $\tau_0,\Delta\tau,\tau_{\max}$
\Ensure  Centroids $\bar{\boldsymbol c}$, assignments $\bar{\boldsymbol z}$
\State $\bar{\boldsymbol c}\leftarrow\bar{\boldsymbol c}^{\,(0)}$;\;
       $\bar{\boldsymbol z}\leftarrow
        \ell_1\text{-}\textsc{Assign}(\bar{\boldsymbol c})$;\;
       $\tau\leftarrow\tau_0$
\While{$\tau\le\tau_{\max}$}
  \Repeat
    \State $\bigl((\boldsymbol{\wh c},\boldsymbol{\wh z}),
           \textit{status}\bigr)\leftarrow$
           Algorithm~\ref{alg:innerloop}$\,
           (\bar{\boldsymbol c},\bar{\boldsymbol z},\tau)$
    \If{$\textit{status}=\textsc{descent}$}
      \State $\bar{\boldsymbol c}\leftarrow\boldsymbol{\wh c}$;\;
             $\bar{\boldsymbol z}\leftarrow\boldsymbol{\wh z}$
    \EndIf
  \Until{$\textit{status}=\textsc{stable}$}
  \State $\tau\leftarrow\tau+\Delta\tau$
\EndWhile
\State \Return $\bar{\boldsymbol c},\bar{\boldsymbol z}$
\end{algorithmic}
\end{algorithm}

\begin{remark}[Partition, sign, and the subproblem] \rm
\label{rem:partition}
$\textsc{Partition}(\bar{\boldsymbol c},\tau)$ refers to the construction of
Section~\ref{subsec:temperature-fixing}: the softmax probabilities
$\rho_{sk}(\tau)$ of~\eqref{eq:softmax}
determine ${\cal J}_1(\tau),{\cal J}_0^{\,\prime}(\tau),
{\cal J}_c^{\,\prime}(\tau)$ and ${\cal J}_s^{\,\prime}(\tau)$, fixing every
binary variable outside ${\cal J}_c^{\,\prime}(\tau)$.  The sign pattern
$\bar\sigma_{sk^{\prime}j}=\operatorname{sign}(\xi^s_j-\bar c^{\,k^{\prime}}_j)$
supplies the linear under-estimates in~\eqref{eq:linearized dc}, which are exact
at $\boldsymbol c=\bar{\boldsymbol c}$. The
subproblem ${\cal P}$ is the instance
of~\eqref{eq:reduced MIP final removed} carried by these index sets.
\end{remark}

\begin{remark}[The assignment map and the descent test] \rm
\label{rem:assign}
$\ell_1\text{-}\textsc{Assign}(\boldsymbol c)$ returns the assignment
constructed in Section~\ref{subsec:working set}: each $\xi^s$ is given an
$\ell_1$-nearest center, ties being resolved by the smaller
$\lVert\xi^s-c^{k}\rVert_p^p$.  The point of applying it before the test
in~\eqref{eq:sufficient-decrease} is that the comparison must be made in the
original problem.  The binary vector $\boldsymbol{\wh z}^{\,w}$ returned
by~\eqref{eq:reduced MIP final removed} need not satisfy the nearest-center
condition of~\eqref{eq:pq}, since the constraints indexed by
${\cal V}(\boldsymbol{\wh c};\tau)$ may be violated, so $v_w$ is in general not
the value of any feasible clustering.  Replacing
$\boldsymbol{\wh z}^{\,w}$ by $\boldsymbol{\wh z}$ restores feasibility, and
both sides of~\eqref{eq:sufficient-decrease} are then values of $F_{p,1}$ at
points feasible to~\eqref{eq:pq}: the test compares the true objective under the
true assignment, not the value reported by the subproblem.
\end{remark}

\begin{remark}[Early termination in the implementation] \rm
\label{rem:early-termination}
Algorithms~\ref{alg:innerloop} and~\ref{alg:pip} carry no computational
budgets.  The implementation imposes two.  The row-generation loop is capped
at $T$ iterations; by Proposition~\ref{prop:working-finite} any
$T\ge|{\cal J}_1(\tau)|+1$ reproduces the behavior analyzed above, while a
smaller $T$ may exit with ${\cal V}(\boldsymbol{\wh c};\tau)\ne\emptyset$.
%The number of descents at each $\tau<\tau_{\max}$ is capped at $r_{\max}$, so that the lower levels serve as a warm-up rather than being driven to stability.%  
Each instance of~\eqref{eq:reduced MIP final removed} is solved
subject to a time limit, and~\eqref{eq:sufficient-decrease} is tested
with a tolerance $\delta = 10^{-6}$, that is, a candidate is accepted when
$F_{p,1}(\boldsymbol{\wh c},\boldsymbol{\wh z})
\le F_{p,1}(\bar{\boldsymbol c},\bar{\boldsymbol z})-\delta$.
\end{remark}

\begin{proposition}[Feasibility of the iterates] \rm
\label{prop:feasible}
Every pair returned by Algorithm~\ref{alg:innerloop} is feasible
to~\eqref{eq:pq}. Consequently every incumbent maintained by
Algorithm~\ref{alg:pip}, and in particular the pair it returns, is feasible
to~\eqref{eq:pq}.
\end{proposition}

\begin{proof}
On the \textsc{descent} exit the returned pair is
$(\boldsymbol{\wh c},\boldsymbol{\wh z})$ with
$\boldsymbol{\wh z}=\ell_1\text{-}\textsc{Assign}(\boldsymbol{\wh c})$, which
assigns each $\xi^s$ to a cluster attaining
$\min_{k\in[K]}\lVert\xi^s-\wh c^{\,k}\rVert_1$ and satisfies
$\sum_k\wh z_{sk}=1$; hence
$(\boldsymbol{\wh c},\boldsymbol{\wh z})$ is feasible to~\eqref{eq:pq}
irrespective of $\boldsymbol{\wh c}$.  On the \textsc{stable} exit the returned
pair is the incumbent.  Algorithm~\ref{alg:pip} initializes the incumbent as
$\bigl(\bar{\boldsymbol c}^{\,(0)},
\ell_1\text{-}\textsc{Assign}(\bar{\boldsymbol c}^{\,(0)})\bigr)$, feasible for
the same reason, and replaces it only by a pair returned with status
\textsc{descent}.  The claim follows by induction on the passes.
\end{proof}

Proposition~\ref{prop:working-finite} concerns a single pass of the
working set loop.  Assembling it with the schedule of
Algorithm~\ref{alg:pip} yields the following termination and optimality
property of the method as a whole.

\begin{proposition}[Fixed point at a stable pass] \rm
\label{prop:fixed-point-pass}
Fix $\tau>0$, let $(\bar{\boldsymbol c},\bar{\boldsymbol z})$ be the incumbent
at which Algorithm~\ref{alg:innerloop} is called, and set
$\bar t_{sk}=\bar z_{sk}\lVert\xi^s-\bar c^{\,k}\rVert_p^p$.  Suppose the call
returns \textsc{stable}, and let
$(\boldsymbol{\wh c},\boldsymbol{\wh z}^{\,w})$ be the solution of its
terminating instance of~\eqref{eq:reduced MIP final removed}, $v_w$ be the optimal
value of that instance, and
$\boldsymbol{\wh z}=\ell_1\text{-}\textsc{Assign}(\boldsymbol{\wh c})$.  If that
instance is solved to global optimality, then
\[
 F_{p,1}(\boldsymbol{\wh c},\boldsymbol{\wh z})\;=\;v_w\;=\;
 v_{\eqref{eq:reduced MIP final}}\;=\;
 F_{p,1}(\bar{\boldsymbol c},\bar{\boldsymbol z}),
\]
and $(\bar{\boldsymbol c},\bar{\boldsymbol z},\bar{\boldsymbol t})$ is globally
optimal for~\eqref{eq:reduced MIP final} at
$\bigl({\cal J}_1(\tau),{\cal J}_0^{\,\prime}(\tau)\bigr)$; that is, statement
{\rm (A)} of Theorem~\ref{th:fixed-point localmin linearized} holds.
\end{proposition}

\begin{proof}
Since $\bar\sigma_{sk^{\prime}j}=\operatorname{sign}(\xi^s_j-\bar c^{\,k^{\prime}}_j)$,
the linearization~\eqref{eq:linearized dc} is exact at
$\boldsymbol c=\bar{\boldsymbol c}$, so
$(\bar{\boldsymbol c},\bar{\boldsymbol z},\bar{\boldsymbol t})$ is feasible for
both~\eqref{eq:reduced MIP final removed} and~\eqref{eq:reduced MIP final} with
objective value $F_{p,1}(\bar{\boldsymbol c},\bar{\boldsymbol z})$.  By
Proposition~\ref{prop:working-finite} the \textsc{stable} exit gives
${\cal V}(\boldsymbol{\wh c};\tau)=\emptyset$, so the pairs of
${\cal J}_1(\tau)$ satisfy the nearest-center inequalities
through~\eqref{eq:working-violations} and the free pairs through the big-$M$
constraints and~\eqref{eq:norm lower bound}; hence
$(\boldsymbol{\wh c},\boldsymbol{\wh z}^{\,w})$ is feasible to~\eqref{eq:pq}
and, by global optimality of the terminating solve,
$v_w=F_{p,1}(\boldsymbol{\wh c},\boldsymbol{\wh z}^{\,w})$.  As
$\boldsymbol{\wh z}$ minimizes $F_{p,1}(\boldsymbol{\wh c},\cdot)$ over the
assignments feasible to~\eqref{eq:pq},
\[
 F_{p,1}(\boldsymbol{\wh c},\boldsymbol{\wh z})\;\le\;
 F_{p,1}(\boldsymbol{\wh c},\boldsymbol{\wh z}^{\,w})\;=\;v_w\;\le\;
 v_{\eqref{eq:reduced MIP final}}\;\le\;
 F_{p,1}(\bar{\boldsymbol c},\bar{\boldsymbol z}),
\]
the third inequality because~\eqref{eq:reduced MIP final} carries all
constraints of~\eqref{eq:reduced MIP final removed} and more.  The
\textsc{stable} status means~\eqref{eq:sufficient-decrease} failed, i.e.
$F_{p,1}(\boldsymbol{\wh c},\boldsymbol{\wh z})\ge
F_{p,1}(\bar{\boldsymbol c},\bar{\boldsymbol z})$, so all four quantities
coincide and $(\bar{\boldsymbol c},\bar{\boldsymbol z},\bar{\boldsymbol t})$
attains the optimal value of~\eqref{eq:reduced MIP final}.
\end{proof}

\begin{proposition}[Convergence and optimality of Algorithm~\ref{alg:pip}] \rm
\label{prop:pip-convergence}
Suppose every solve of~\eqref{eq:reduced MIP final removed} returns a point that
is feasible for its instance and depends only on that instance.  Then
Algorithm~\ref{alg:pip} terminates with its last call to
Algorithm~\ref{alg:innerloop} returning \textsc{stable} at $\tau=\tau_{\max}$,
and its output $(\bar{\boldsymbol c},\bar{\boldsymbol z})$ is feasible
to~\eqref{eq:pq}.  If moreover the terminating instance is solved to global
optimality, then $(\bar{\boldsymbol c},\bar{\boldsymbol z},
\bar{\boldsymbol t})$, with
$\bar t_{sk}=\bar z_{sk}\lVert\xi^s-\bar c^{\,k}\rVert_p^p$, is globally optimal
for~\eqref{eq:reduced MIP final} at
$\bigl({\cal J}_1(\tau_{\max}),{\cal J}_0^{\,\prime}(\tau_{\max})\bigr)$, so
that statement {\rm (A)} of
Theorem~\ref{th:fixed-point localmin linearized} holds; and if
$\bar c^{\,k}_j\neq\xi^s_j$ for every $(s,k,j)\in[N]\times[K]\times[d]$, then
$(\bar{\boldsymbol c},\bar{\boldsymbol z})$ is a locally optimal solution
of~\eqref{eq:pq}.
\end{proposition}

\begin{proof}
Every call to Algorithm~\ref{alg:innerloop} terminates by
Proposition~\ref{prop:working-finite}, returning \textsc{descent} or
\textsc{stable}.  Fix a level $\tau$.  An instance
of~\eqref{eq:reduced MIP final removed} arising there is determined by
$\bigl({\cal J}_1(\tau),{\cal J}_0^{\,\prime}(\tau),
{\cal J}_c^{\,\prime}(\tau),{\cal J}_s^{\,\prime}(\tau),\bar\sigma,
{\cal J}_w(\tau)\bigr)$, each component ranging over a finite set, so only
finitely many instances occur; since the returned point depends only on the
instance, the incumbents reachable at that level form a finite set, on which
$F_{p,1}$ takes finitely many values.  A call returning \textsc{descent}
strictly decreases $F_{p,1}(\bar{\boldsymbol c},\bar{\boldsymbol z})$, so only
finitely many do, and the level ends with a \textsc{stable} call, after which
$\tau$ is advanced by $\Delta\tau$.  As the schedule visits at most
$\lfloor(\tau_{\max}-\tau_0)/\Delta\tau\rfloor+1$ levels, the algorithm
returns, the last call executed being \textsc{stable} at $\tau=\tau_{\max}$.
Feasibility of the output is obtained by Proposition~\ref{prop:feasible}, and applying
Proposition~\ref{prop:fixed-point-pass} to that call gives statement {\rm (A)};
the final claim is {\rm (A)}\,$\Rightarrow$\,{\rm (C)} of
Theorem~\ref{th:fixed-point localmin linearized}.
\end{proof}

\section{Numerical Experiments}
\label{sec: numerical-results}
In this section, we present extensive numerical comparisons to demonstrate the performance of the mixed-norm clustering model and the efficiency of the PIP method. All experiments are run on an Apple M2 MacBook Air (8\,GB RAM)
using {\sc Gurobi}~12.0 as the MIP solver (academic license).
Code is implemented in Python~3 with necessary packages.
Each IP subproblem is solved with a per-call time limit; running times
reported below include all solver overhead.

\subsection{Compared methods, experimental protocol, 
and evaluation metrics} \label{sec:experiment-setup}

To our knowledge, no alternating-minimization scheme has been 
proposed for the mixed-norm ($p\neq q$) clustering objective. 
We therefore introduce a natural
benchmark that transplants the spirit of Lloyd's algorithm 
\cite{Lloyd1982}
($K$-means) and of $K$-medians: alternate between (i) assigning each point to its nearest center
in the assignment norm $\ell_q$, and (ii) recomputing each center as the
optimal cluster representative in the objective norm $\ell_p$. We call it
Alter $L_p/L_q$; e.g.\ Alter $L_2/L_1$ denotes the $p{=}2,\,q{=}1$, and Alter $L_1/L_1$ recovers $k$-medians.

\paragraph{Scheme.}
Starting from initial centers $\mathbf c^{(0)}$ (e.g.\ $k$-means$++$), iterate the
two maps
\begin{align}
  \text{(assignment $\ell_q$)}\quad
  &z_{sk}^{(t)} = 1 \iff
     k \in \displaystyle{
     \operatornamewithlimits{\mbox{\bf argmin}}_{j\in[K]} } \, \bigl\|\xi^s-c^{j,(t)}\bigr\|_q,
     \label{eq:Estep}\\[2pt]
  \text{(update $\ell_p$)}\quad
  &c^{k,(t+1)} = \mu_p\!\Bigl(\{\,\xi^s : z^{(t)}_{sk}=1\,\}\Bigr)
   \;\triangleq \, \displaystyle{
   \operatornamewithlimits{\mbox{\bf argmin}}_{c\in\mathbb{R}^d}
   } \displaystyle{
   \sum_{s:\,z^{(t)}_{sk}=1}
   } \, \bigl\| \, \xi^s- c \, \bigr\|_p^{\,p},
     \label{eq:Mstep}
\end{align}
with ties in \eqref{eq:Estep} broken by a fixed rule 
(e.g. the smallest $k$).

\gap

Here, we want to mention that the iteration of the above 
alternating algorithm may enter a cycle, and convergence even 
to a fixed point is not guaranteed. Because of this, we treat 
Alter $L_p/L_q$ purely as a fast, easily implemented
baseline and warm start, not as a solver. We run it for at most 
$T_{\max} = 100$
iterations, terminating early when the labels are unchanged, 
and return the visited iterate of lowest objective to make the
output well-defined even when descent fails. 
%Each iteration costs $O(NKP)$, so the scheme is cheap. 
It also provides the initialization for our progressive
integer programming (PIP) method.

In the experiments, we compare six methods, including all four 
norm pairs plus two MIP models with $q = 1$ solved by the 
proposed PIP method.  Every alternating method is an instance 
of the scheme introduced above. We summarize the six methods in
the table below:

\begin{center}
\small
\begin{tabular}{l l l l}
\toprule
Method & Assignment ($q$) & Centroid update ($p$) & MIP counterpart \\
\midrule
K-means            & $\ell_2$ & mean ($p{=}2$)   & --- \\
K-medians          & $\ell_1$ & median ($p{=}1$) & PIP ($p{=}q{=}1$) \\
Alter $L_1/L_2$ ($p{=}1,q{=}2$) & $\ell_2$ & median ($p{=}1$) & --- \\
Alter $L_2/L_1$ ($p{=}2,q{=}1$) & $\ell_1$ & mean ($p{=}2$)   & PIP ($p{=}2,q{=}1$) \\
\bottomrule
\end{tabular}
\end{center}

\noindent {\bf Seeds and initialization.}
Each experiment is repeated over five seeds; unless stated 
otherwise, the
standard seed set $\{42, 7, 99, 123, 2024\}$ is used, and the adversarial experiments of Section~\ref{sec:adversarial-initialization} use the dedicated set $\{3, 20, 46, 53, 71\}$ (see the disclosure there).  The seed drives both the data draw and the initialization, so every number below is reproducible.  For each seed, one set of $K$ initial centroids is produced by K-Means++ (scikit-learn's \texttt{kmeans\_plusplus} with
\texttt{random\_state} equal to the seed) or by fixed adversarial
centroids where stated, and given identically to all alternating
methods.

\gap

\noindent {\bf PIP and solver settings.}
PIP ($p{=}q{=}1$) and PIP ($p{=}2,q{=}1$) solve the corresponding
mixed-integer programs by the progressive method of
Section \ref{sec: PIP-q=1}, where the restricted subproblems are 
solved by {\sc Gurobi}.  The cases for $p =1$ and $p = 2$ 
share the $\ell_1$ assignment constraints ($q{=}1$) and differ 
only in the dispersion term (linear for $p{=}1$, quadratic for 
$p=2$).  Each PIP is warm-started
from the alternating method that optimizes the \emph{same} 
objective, so the PIP method will only improve its own objective 
from the heuristic's solution.  In this paper, we focus on the 
cases where $q = 1$, and the
case for $p = 1$ and $q = 2$ is not included yet.
Moreover, unless stated otherwise, all PIP runs use identical 
settings: the softmax
temperature $\tau$ from $0.1$ to $0.5$ in steps of $0.1$, with 
at most $30$
row-generation iterations per level, and a time limit of $30$\,s is set for {\sc Gurobi} to solve a single subproblem.  The only exception is the stress test with large-$K$ in Section \ref{sec:largeK}, where each subproblem is two orders of magnitude larger and we use a reduced budget of $10$ iterations per level with a $60$\,s per-solve limit.

\gap

\noindent {\bf Evaluation metrics.} We evaluate the performance of 
the clustering model employing:

\gap

\noindent \textbf{ARI} (adjusted Rand index) measures agreement 
between the recovered partition and the ground-truth partition, 
corrected for chance: $1$ for a perfect match, $\approx 0$ 
for a random assignment.  It is our primary external metric.

\gap

\noindent $\mathbf{WCD_1} = \sum_{s} \| \xi^s - c^{k(s)} \|_1$, 
the within-cluster distortion in $\ell_1$, where
$c_{k(s)}$ is the centroid to which the method assigns point $s$.
This is the internal objective of the $p{=}1$ family.

\gap

\noindent $\mathbf{WCD_2} = \sum_{s} \| \xi^s - c^{k(s)} \|_2^2$, 
the within-cluster distortion in squared $\ell_2$-norm.  This is
the internal objective of the $p = 2$ family.

\gap

\noindent By default, all metrics are computed over all $N$ points. 
For contaminated datasets, they are computed on the non-outliers, as stated in the
corresponding section.  Reporting both distortions for every method
enables cross-norm comparisons: in particular, whether a method is
inferior at its own objective. 
%The objective values and ARI are seed-dependent; they are reproducible for a given seed but the computational times may vary with for  instance machine load.  So the reported times are only indicative of the relative computational costs. 

\subsection{The performance of the mixed-norm clustering model on noisy data}

In this section, we will demonstrate the advantages of the mixed-norm clustering model \eqref{eq:pq} with $p \neq q$ on noisy data with outliers. The   $\ell_{p,q}$ formulation benefits the data structure when
assignment geometry and the center estimator can be matched separately to different norms, which is shown in the next two experiments.  In general, no single norm pair is expected to dominate across all data settings.  

\subsubsection{Synthetic data with coordinate-sparse outliers}
\label{sec:sparse-corruption}

In this section, we demonstrate the superior performance of the $\ell_{2, 1}$ mixed-norm clustering model on synthetic data with coordinate-sparse outliers. We generate $K=3$ isotropic Gaussian clusters with $d=40$.  The
centers $\mu_1, \mu_2, \mu_3$ are drawn once from
$\mu_k\sim\mathcal{N}(0,\sigma^2I_{40})$, with
$\sigma=0.6$ and a fixed center seed.  For each run seed and each
cluster, we draw $n=30$ points from $\mathcal{N}(\mu_k,I_{40})$ and corrupt a
fraction $\rho=0.30$ of them.  Each corrupted point receives one spike with values 
$\pm 7$, on a uniformly selected coordinate.  The sign is
symmetric, so the coordinate mean remains unbiased.  Thus only one of the
$40$ coordinates is affected for each outlier, and the per-coordinate
contamination rate is $\rho/d\approx0.75\%$.

This experimental design separates the two roles of the norms.  For assignment, $\ell_1$ distance is more robust to outliers by aggregating the spike with the remaining $39$ clean
coordinates, which thus favors $q=1$.  For center estimation, each coordinate is
contaminated only rarely and symmetrically, so the sample mean remains an
efficient estimator, whereas the coordinatewise median incurs its usual
efficiency loss; this favors $p=2$. Therefore, this indicates that the
mixed norm pair $(p,q)=(2,1)$ may be more favorable under this data generation setting with coordinate-sparse outliers.

All four alternating methods use the same 
$K$-means++ initialization in each
run.  We report metrics both on the full sample and on 
the uncontaminated observations.  The latter, denoted by
``clean,'' isolate recovery of the latent
cluster structure from the unavoidable contribution of the 
spikes to the full-sample distortion.  
Table~\ref{tab:claim1prime} reports means and standard
deviations over $80$ runs.

\begin{table}[!htbp]
\begin{center}
\small
\caption{Sparse high-dimensional corruption ($K=3$, $d=40$, $n=30$/cluster,
$\mathrm{sep}=0.6$; single $\pm 7$ spike on one of $40$ coordinates for $30\%$ of points. ) }
\label{tab:claim1prime}
\setlength{\tabcolsep}{5pt}
\begin{tabular}{l c r r r r}
\toprule
Method & ARI $\uparrow$ & WCD$_2$ $\downarrow$ & WCD$_1$ 
$\downarrow$ & WCD$_2$ (clean) $\downarrow$ & WCD$_1$ (clean) 
$\downarrow$ \\
\midrule
$K$-means \hfill ($p{=}q{=}2$)   & $0.664{\pm}0.24$ & $4872.7{\pm}257$ & $3065.7{\pm}75$ & $2596.7{\pm}230$ & $2050.8{\pm}149$ \\
$K$-medians \hfill ($p{=}q{=}1$) & $0.714{\pm}0.24$ & $4952.4{\pm}252$ & $\mathbf{2998.4}{\pm}69$ & $2583.2{\pm}222$ & $\mathbf{2000.7}{\pm}145$ \\
\textsc{Alter} \hfill ($p{=}1,q{=}2$)   & $0.683{\pm}0.22$ & $4949.9{\pm}258$ & $3024.6{\pm}79$ & $2605.9{\pm}225$ & $2020.0{\pm}145$ \\
\textsc{Alter} \hfill ($p{=}2,q{=}1$)   & $\mathbf{0.766}{\pm}0.23$ & $\mathbf{4830.0}{\pm}252$ & $3029.2{\pm}71$ & $\mathbf{2550.9}{\pm}217$ & $2026.6{\pm}146$ \\
\bottomrule
\end{tabular}
\end{center}

\vspace{-0.1in}

\noindent 
{\small {\bf Note:} Each column reports mean $\pm$ std over $80$
runs ($4$ center geometries $\times$ $20$ seeds).  ``(clean)'' 
columns exclude the spiked points.  Best value per column is 
highlighted in bold.}
\end{table}

\noindent The numerical results agree with our prediction.  
The $\ell_{2,1}$ method
has the highest mean ARI, the smallest full-sample 
$\mathrm{WCD}_2$, and the
smallest clean $\mathrm{WCD}_2$.  $K$-medians retains the smallest
$\mathrm{WCD}_1$, as expected from its fitting loss, but gives 
a lower mean ARI. 
%These results demonstrate the flexibility of the proposed model \eqref{eq:pq} and a performance gain with an appropriate choice of $p, q \in \{1, 2\}$.

\subsubsection{Synthetic data with dense heavy-tailed contamination}
\label{sec:dense-corruption}

Next, we consider the noisy data with dense heavy-tailed contamination. Consider $K=6$ and the cluster
centers $\mu_1,\ldots,\mu_6\in\mathbb{R}^6$ are drawn once as
$\mu_k\sim\mathcal{N}(0,4^2I_6)$ using random seed $0$ and are shared
across all runs.  For every random seed $r\in\{42,7,99,123,2024\}$ and every cluster $k$,
we generate
\begin{enumerate}
 \item $n_{\mathrm{in}}=50$ uncontaminated observations from a normal distribution
 $\mathcal{N}(\mu_k,0.8^2I_6)$; and
 \item $n_{\mathrm{out}}=16$ contaminated observations of the form
 $\mu_k+2.5z$, where the six coordinates of $z$ are independently generated standard
 Cauchy variables.
\end{enumerate}
Pooling and permuting the observations gives $N=396$, of which
$96/396\approx24\%$ are contaminated.  All methods are fitted to the complete
sample, whereas ARI and the reported distortions in
Table~\ref{tab:claim1} are evaluated on the uncontaminated observations. Here the contamination affects every coordinate and directly challenges the
center estimator.  Euclidean assignment remains aligned with the spherical
cluster geometry, while the coordinatewise median protects the fitted centers
from the Cauchy tails.  Therefore, the mixed pair $(p,q)=(1,2)$ will be favorable based on the above discussion.

We can observe the consistent performance from the numerical results. $K$-means combines the desired
assignment geometry with a nonrobust mean and attains an ARI of $0.524$ on average.
$K$-medians improves robustness but changes the assignment geometry, reaching a higher
ARI of $0.750$ on average.  The $\ell_{1,2}$ mixed-norm clustering combines Euclidean assignment and
median fitting. It attains the highest average ARI of $0.805$, and the smallest
average $\mathrm{WCD}_2$ of $3797.0$, among the compared methods.  The smallest
$\mathrm{WCD}_1$ of $1709.6$, is obtained by PIP for the matched
$\ell_{1,1}$ model, which is consistent with that model's target loss.

The IP refinements also separate modeling error from optimization error.  PIP
raises the average ARI of $K$-medians from $0.750$ to $0.803$, bringing the
matched $\ell_1$ model close to the mixed model but not beyond its average ARI.
For $(p,q)=(2,1)$, PIP reduces the reported average $\mathrm{WCD}_2$ from
$12983.5$ to $12702.8$, yet the resulting ARI remains in the bottom tier.
Thus the poor performance of $\ell_{2,1}$ on this dense heavy-tailed design is
primarily statistical rather than a consequence of an inferior alternating
solution.  Because the table evaluates only uncontaminated observations while
PIP accepts steps using the full-sample objective, clean-sample distortion is
not required to decrease on every individual seed.

\definecolor{pairshade}{gray}{0.92}
\newcommand{\pr}{\cellcolor{pairshade}}

\begin{table}[t]
{\begin{center}
\caption{Contaminated spherical clusters: $(K,d,N,\varepsilon) =
(6,6,396,24\%)$
% ($K=6$, $d=6$, $N=396$, $\varepsilon\approx24\%$ 
within-cluster Cauchy noise.}
\label{tab:claim1}
\small
\setlength{\tabcolsep}{3pt}
\renewcommand{\arraystretch}{0.95}
\begin{tabular}{c c l r r r @{\hspace{1.2em}} c c l r r r}
\toprule
Seed & $(p,q)$ & Method & ARI & $\mathrm{WCD}_1$ & $\mathrm{WCD}_2$ &
Seed & $(p,q)$ & Method & ARI & $\mathrm{WCD}_1$ & $\mathrm{WCD}_2$ \\
\midrule
\multirow{6}{*}{42}
 & $(2,2)$ & $K$-means & 0.681 & 2980.2 & 7338.9 &
\multirow{6}{*}{123}
 & $(2,2)$ & $K$-means & 0.330 & 4060.5 & 14072.8 \\
\cmidrule(lr){2-6}\cmidrule(l){8-12}
 & \multirow{2}{*}{$(1,1)$} & $K$-medians & 0.566 & 2281.0 & 7285.4 &
 & \multirow{2}{*}{$(1,1)$} & $K$-medians & 0.566 & 2282.1 & 7299.5 \\
 & & PIP & 0.821 & 1704.7 & 3645.5 &
 & & PIP & 0.566 & 2282.1 & 7299.5 \\
\cmidrule(lr){2-6}\cmidrule(l){8-12}
 & $(1,2)$ & \textsc{Alter} & 0.821 & 1800.6 & 3586.1 &
 & $(1,2)$ & \textsc{Alter} & 0.563 & 2295.8 & 7217.6 \\
\cmidrule(lr){2-6}\cmidrule(l){8-12}
 & \multirow{2}{*}{$(2,1)$} & \textsc{Alter} & 0.567 & 3151.1 & 8521.5 &
 & \multirow{2}{*}{$(2,1)$} & \textsc{Alter} & 0.327 & 3906.6 & 14256.8 \\
 & & PIP & 0.561 & 2746.6 & 7117.9 &
 & & PIP & 0.327 & 3906.6 & 14256.8 \\
\midrule
\multirow{6}{*}{7}
 & $(2,2)$ & $K$-means & 0.567 & 3861.9 & 18747.2 &
\multirow{6}{*}{2024}
 & $(2,2)$ & $K$-means & 0.474 & 4411.6 & 16211.7 \\
\cmidrule(lr){2-6}\cmidrule(l){8-12}
 & \multirow{2}{*}{$(1,1)$} & $K$-medians & 0.821 & 1697.2 & 3351.6 &
 & \multirow{2}{*}{$(1,1)$} & $K$-medians & 1.000 & 1150.7 & 1192.8 \\
 & & PIP & 0.813 & 1692.9 & 3809.4 &
 & & PIP & 1.000 & 1150.7 & 1192.8 \\
\cmidrule(lr){2-6}\cmidrule(l){8-12}
 & $(1,2)$ & \textsc{Alter} & 0.821 & 1780.5 & 3502.1 &
 & $(1,2)$ & \textsc{Alter} & 1.000 & 1148.5 & 1185.5 \\
\cmidrule(lr){2-6}\cmidrule(l){8-12}
 & \multirow{2}{*}{$(2,1)$} & \textsc{Alter} & 0.681 & 3531.6 & 12187.0 &
 & \multirow{2}{*}{$(2,1)$} & \textsc{Alter} & 0.330 & 4324.8 & 15713.2 \\
 & & PIP & 0.681 & 3531.6 & 12187.0 &
 & & PIP & 0.330 & 4324.8 & 15713.2 \\
\midrule
\multirow{6}{*}{99}
 & $(2,2)$ & $K$-means & 0.567 & 3625.9 & 13402.1 &
\multirow{6}{*}{avg}
 & $(2,2)$ & $K$-means & 0.524 & 3788.0 & 13954.5 \\
\cmidrule(lr){2-6}\cmidrule(l){8-12}
 & \multirow{2}{*}{$(1,1)$} & $K$-medians & 0.800 & 1717.0 & 4035.9 &
 & \multirow{2}{*}{$(1,1)$} & $K$-medians & 0.750 & 1825.6 & 4633.1 \\
 & & PIP & 0.813 & 1717.5 & 3845.5 &
 & & PIP & 0.803 & \textbf{1709.6} & 3958.5 \\
\cmidrule(lr){2-6}\cmidrule(l){8-12}
 & $(1,2)$ & \textsc{Alter} & 0.821 & 1716.7 & 3493.8 &
 & $(1,2)$ & \textsc{Alter} & \textbf{0.805} & 1748.4 & \textbf{3797.0} \\
\cmidrule(lr){2-6}\cmidrule(l){8-12}
 & \multirow{2}{*}{$(2,1)$} & \textsc{Alter} & 0.413 & 4027.3 & 14239.2 &
 & \multirow{2}{*}{$(2,1)$} & \textsc{Alter} & 0.464 & 3788.3 & 12983.5 \\
 & & PIP & 0.413 & 4027.3 & 14239.2 &
 & & PIP & 0.462 & 3707.4 & 12702.8 \\
\bottomrule
\end{tabular}
\end{center}}

\vspace{-0.1in}

\noindent 
{\small {\bf Note:} ARI, $\mathrm{WCD}_1$, and $\mathrm{WCD}_2$ 
are evaluated on the
uncontaminated observations.  The $(p,q)$ column identifies the 
norm pair; an entry spanning two rows marks an 
alternating--PIP pair sharing the same
objective, whereas $(2,2)$ and $(1,2)$ have no PIP counterpart 
in this table.  Averaged over the five seeds, PIP $(p{=}q{=}1)$ 
requires
$242.4$\,s and PIP $(p{=}2,q{=}1)$ requires $49.4$\,s per run; the
alternating methods terminate in less than $0.01$\,s.}
\end{table}

\subsection{IP can escape bad fixed points of the alternating 
algorithm}

\label{sec:escaping-alternating}

The preceding experiments show that choosing an appropriate norm pair matters.
We next hold the norm pair fixed and ask whether integer programming can improve
the solution returned by alternating optimization.  We begin with controlled
adversarial initializations, for which the alternating methods deliberately
enter inferior basins, and then consider a large-$K$ experiment in which the
same phenomenon arises under ordinary $K$-means++ initialization.

\subsubsection{Controlled stress tests with adversarial initialization}
\label{sec:adversarial-initialization}

Tables~\ref{tab:gaussian_adv} and~\ref{tab:cauchy_adv} use prescribed adverse
initial centers on, respectively, a four-component Gaussian mixture and a
four-component truncated-Cauchy mixture.  Both designs share the same four cluster centers,
$\mu_1=(-8,0)$, $\mu_2=(2,8)$, $\mu_3=(2,-6)$, and $\mu_4=(12,0)$ in
$\mathbb{R}^2$, and differ only in the noise distribution.  For the Gaussian
mixture we draw $N=600$ observations, $150$ per cluster, from
$\mathcal{N}(\mu_k, 1.5^2 I_2)$.  For the truncated-Cauchy mixture we draw
$N=200$ observations, $50$ per cluster: each coordinate is generated as
$\mu_{k}+0.2\tan\!\big(\pi(u-\tfrac12)\big)$ with $u\sim\mathrm{Unif}(0.01,0.99)$,
that is, a standard Cauchy variable truncated to its central $98\%$ and scaled
by $0.2$.  The truncation keeps the sample bounded while preserving the heavy
tails that distinguish this design from the Gaussian one.

In place of $K$-means++, all methods receive a fixed set of
adverse initial
centers: \\ $\{(-0.5,9),(-2,-4),(10.5,0),(13.5,0)\}$ for the 
Gaussian mixture and
$\{(-1,7),(0,-4),(10,0),(14,0)\}$ for the Cauchy mixture.  
In both cases two
initial centers are placed near the same true cluster while
another true
cluster is left uncovered, so escaping requires an alternating 
method to move a
redundant center across the configuration.  The five seeds
$\{3,20,46,53,71\}$ control the data draw only.  The purpose is 
diagnostic:
because the norm pairs largely agree on the well-separated target 
partition, a
failure to recover it isolates the effect of the optimization 
method rather
than the modeling choice.  Each PIP run is warm-started from 
the corresponding alternating solution.

\begin{table}[t]
{\begin{center}
\caption{Four-component Gaussian mixture 
($K=4$, $N=600$, $d=2$) under adversarial initialization. }
\label{tab:gaussian_adv}
\small
\setlength{\tabcolsep}{3pt}
\renewcommand{\arraystretch}{0.95}
\begin{tabular}{c c l r r r @{\hspace{1.2em}} c c l r r r}
\toprule
Seed & $(p,q)$ & Method & ARI & $\mathrm{WCD}_1$ & $\mathrm{WCD}_2$ &
Seed & $(p,q)$ & Method & ARI & $\mathrm{WCD}_1$ & $\mathrm{WCD}_2$ \\
\midrule
\multirow{5}{*}{3}
 & $(2,2)$ & $K$-means & 0.627 & 2993.3 & 11991.7 &
\multirow{5}{*}{53}
 & $(2,2)$ & $K$-means & 0.621 & 3029.9 & 12406.3 \\
\cmidrule(lr){2-6}\cmidrule(l){8-12}
 & \multirow{2}{*}{$(1,1)$} & $K$-medians & 0.625 & 2986.8 & 12145.8 &
 & \multirow{2}{*}{$(1,1)$} & $K$-medians & 0.622 & 3010.7 & 13130.1 \\
 & & PIP & 0.625 & 2986.8 & 12145.8 &
 & & PIP & 0.622 & 3010.7 & 13130.1 \\
\cmidrule(lr){2-6}\cmidrule(l){8-12}
 & \multirow{2}{*}{$(2,1)$} & \textsc{Alter} & 0.627 & 2993.7 & 11992.4 &
 & \multirow{2}{*}{$(2,1)$} & \textsc{Alter} & 0.627 & 3019.0 & 12445.6 \\
 & & PIP & 1.000 & 1453.0 & 2722.3 &
 & & PIP & 0.996 & 1460.2 & 2780.0 \\
\midrule
\multirow{5}{*}{20}
 & $(2,2)$ & $K$-means & 0.623 & 2997.0 & 11916.1 &
\multirow{5}{*}{71}
 & $(2,2)$ & $K$-means & 0.627 & 2975.6 & 11781.4 \\
\cmidrule(lr){2-6}\cmidrule(l){8-12}
 & \multirow{2}{*}{$(1,1)$} & $K$-medians & 0.623 & 2981.9 & 12961.2 &
 & \multirow{2}{*}{$(1,1)$} & $K$-medians & 0.627 & 2967.9 & 12149.0 \\
 & & PIP & 0.623 & 2981.9 & 12961.2 &
 & & PIP & 1.000 & 1379.7 & 2499.8 \\
\cmidrule(lr){2-6}\cmidrule(l){8-12}
 & \multirow{2}{*}{$(2,1)$} & \textsc{Alter} & 0.627 & 2988.8 & 11934.5 &
 & \multirow{2}{*}{$(2,1)$} & \textsc{Alter} & 0.627 & 2975.0 & 11781.7 \\
 & & PIP & 1.000 & 1433.8 & 2613.0 &
 & & PIP & 1.000 & 1381.1 & 2494.3 \\
\midrule
\multirow{5}{*}{46}
 & $(2,2)$ & $K$-means & 0.619 & 3036.5 & 12341.3 &
\multirow{5}{*}{avg}
 & $(2,2)$ & $K$-means & 0.623 & 3006.4 & 12087.4 \\
\cmidrule(lr){2-6}\cmidrule(l){8-12}
 & \multirow{2}{*}{$(1,1)$} & $K$-medians & 0.621 & 3015.1 & 13338.1 &
 & \multirow{2}{*}{$(1,1)$} & $K$-medians & 0.624 & 2992.5 & 12744.8 \\
 & & PIP & 1.000 & 1441.1 & 2803.1 &
 & & PIP & 0.774 & 2360.0 & 8708.0 \\
\cmidrule(lr){2-6}\cmidrule(l){8-12}
 & \multirow{2}{*}{$(2,1)$} & \textsc{Alter} & 0.625 & 3026.2 & 12369.5 &
 & \multirow{2}{*}{$(2,1)$} & \textsc{Alter} & 0.627 & 3000.5 & 12104.7 \\
 & & PIP & 1.000 & 1446.6 & 2780.0 &
 & & PIP & \textbf{0.999} & \textbf{1434.9} & \textbf{2677.9} \\
\bottomrule
\end{tabular}
\end{center}}

\vspace{-0.1in}

\noindent 
{\small {\bf Note:}  ARI, $\mathrm{WCD}_1$, and
$\mathrm{WCD}_2$ use all observations.  The $(p,q)$ column 
identifies the norm pair; an entry spanning two rows marks an 
alternating--PIP pair sharing the same objective, 
whereas $(2,2)$ has no PIP counterpart.  Averaged over the five
seeds, PIP $(p{=}q{=}1)$ requires $66.2$\,s and PIP 
$(p{=}2,q{=}1)$ requires
$25.2$\,s per run; the alternating methods terminate in less than 
$0.01$\,s.} 
\end{table}

On the Gaussian instance, the three alternating methods 
remain near the same inferior partition on every seed, 
with ARI between $0.619$ and $0.627$.  PIP
for $(p,q)=(2,1)$ changes this outcome decisively: its 
average ARI is $0.999$,
and it attains ARI $1.000$ on four seeds and $0.996$ on the remaining seed.
PIP for $(p,q)=(1,1)$ also escapes on seeds $46$ and $71$, but not on the
other three.  The contrast shows that an alternating fixed point can be far
from the best solution accessible to the same model, and that a structured IP
search can cross the assignment barrier responsible for the failure.

\begin{table}[t]
\begin{center}
\caption{Four-component Cauchy mixture ($K=4$, $N=200$, $d=2$) under adversarial initialization.}
\label{tab:cauchy_adv}
\small
\setlength{\tabcolsep}{3pt}
\renewcommand{\arraystretch}{0.95}
\begin{tabular}{c c l r r r @{\hspace{1.2em}} c c l r r r}
\toprule
Seed & $(p,q)$ & Method & ARI & $\mathrm{WCD}_1$ & $\mathrm{WCD}_2$ &
Seed & $(p,q)$ & Method & ARI & $\mathrm{WCD}_1$ & $\mathrm{WCD}_2$ \\
\midrule
\multirow{5}{*}{3}
 & $(2,2)$ & $K$-means & 1.000 & 165.6 & 274.2 &
\multirow{5}{*}{53}
 & $(2,2)$ & $K$-means & 0.681 & 903.6 & 3788.3 \\
\cmidrule(lr){2-6}\cmidrule(l){8-12}
 & \multirow{2}{*}{$(1,1)$} & $K$-medians & 0.628 & 887.0 & 3761.3 &
 & \multirow{2}{*}{$(1,1)$} & $K$-medians & 0.606 & 902.4 & 5863.6 \\
 & & PIP & 1.000 & 159.6 & 278.1 &
 & & PIP & 0.665 & 894.2 & 5016.8 \\
\cmidrule(lr){2-6}\cmidrule(l){8-12}
 & \multirow{2}{*}{$(2,1)$} & \textsc{Alter} & 1.000 & 165.6 & 274.2 &
 & \multirow{2}{*}{$(2,1)$} & \textsc{Alter} & 0.666 & 909.7 & 3804.9 \\
 & & PIP & 1.000 & 165.6 & 274.2 &
 & & PIP & 0.681 & 903.6 & 3788.3 \\
\midrule
\multirow{5}{*}{20}
 & $(2,2)$ & $K$-means & 0.671 & 875.8 & 3564.4 &
\multirow{5}{*}{71}
 & $(2,2)$ & $K$-means & 0.665 & 880.7 & 3613.8 \\
\cmidrule(lr){2-6}\cmidrule(l){8-12}
 & \multirow{2}{*}{$(1,1)$} & $K$-medians & 0.626 & 873.2 & 3589.8 &
 & \multirow{2}{*}{$(1,1)$} & $K$-medians & 0.621 & 870.0 & 4066.9 \\
 & & PIP & 0.698 & 871.2 & 3667.3 &
 & & PIP & 1.000 & 177.2 & 343.6 \\
\cmidrule(lr){2-6}\cmidrule(l){8-12}
 & \multirow{2}{*}{$(2,1)$} & \textsc{Alter} & 0.681 & 874.5 & 3563.1 &
 & \multirow{2}{*}{$(2,1)$} & \textsc{Alter} & 0.675 & 878.2 & 3609.9 \\
 & & PIP & 0.692 & 873.4 & 3561.9 &
 & & PIP & 0.698 & 873.4 & 3596.7 \\
\midrule
\multirow{5}{*}{46}
 & $(2,2)$ & $K$-means & 0.658 & 900.5 & 3712.1 &
\multirow{5}{*}{avg}
 & $(2,2)$ & $K$-means & 0.735 & 745.2 & 2990.6 \\
\cmidrule(lr){2-6}\cmidrule(l){8-12}
 & \multirow{2}{*}{$(1,1)$} & $K$-medians & 0.618 & 893.9 & 5058.4 &
 & \multirow{2}{*}{$(1,1)$} & $K$-medians & 0.620 & 885.3 & 4468.0 \\
 & & PIP & 1.000 & 166.8 & 268.7 &
 & & PIP & \textbf{0.873} & \textbf{453.8} & \textbf{1914.9} \\
\cmidrule(lr){2-6}\cmidrule(l){8-12}
 & \multirow{2}{*}{$(2,1)$} & \textsc{Alter} & 1.000 & 171.8 & 265.5 &
 & \multirow{2}{*}{$(2,1)$} & \textsc{Alter} & 0.804 & 600.0 & 2303.5 \\
 & & PIP & 1.000 & 171.8 & 265.5 &
 & & PIP & 0.814 & 597.6 & 2297.3 \\
\bottomrule
\end{tabular}
\end{center}

\vspace{-0.1in}

{\small
{\bf Note:} Truncated
Cauchy scales $0.2$. ARI, $\mathrm{WCD}_1$, and $\mathrm{WCD}_2$ 
use all observations.  The $(p,q)$
column identifies the norm pair; an entry spanning two rows marks 
an alternating--PIP pair sharing the same objective, whereas 
$(2,2)$ has no PIP counterpart.  Averaged over the five seeds, 
PIP $(p{=}q{=}1)$ requires
$62.1$\,s and PIP $(p{=}2,q{=}1)$ requires $0.5$\,s per run; the 
alternating methods terminate in less than $0.01$\,s.}
\end{table}

The Cauchy experiment gives the same optimization message with 
a different
preferred fitting loss.  $K$-medians is trapped on every seed and 
has average ARI $0.620$.  PIP for $(p,q)=(1,1)$ reaches ARI 
$1.000$ on seeds $3$, $46$,
and $71$, raising the average to $0.873$ and reducing the average
$\mathrm{WCD}_1$ from $885.3$ to $453.8$.  PIP for $(p,q)=(2,1)$ produces a
more modest average increase from ARI $0.804$ to $0.814$.  The different
outcomes across the Gaussian and Cauchy experiments reinforce the two-level
interpretation of the proposed framework: the norm pair determines the
statistical landscape, while the IP refinement determines whether a poor
alternating basin can be escaped within the computational budget.

These runs should not be interpreted as global certificates for the original
clustering problem merely because an improvement is found.  The relevant
theoretical certificate applies when the self-generated restricted subproblem
is solved globally and the fixed-point condition is satisfied.  The empirical
conclusion here is narrower and directly supported by the tables: warm-started
PIP can move from an inferior alternating solution to a substantially better
feasible clustering.

\subsubsection{Large \texorpdfstring{$K$}{K} under ordinary
\texorpdfstring{$K$-means++}{K-means++} initialization}
\label{sec:largeK}

The adversarial experiments establish that alternating methods can be trapped,
but they do not show how often this matters under a standard initialization.
We therefore consider a small-$N$/large-$K$ grid of isotropic Gaussian
clusters.  The $K=25$ centers are placed on a regular $5\times5$ grid in
$\mathbb{R}^2$ with spacing $D=5\sigma$, where $\sigma=1$, and each center
coordinate is then displaced by an independent draw from
$\mathrm{Unif}(-0.2\sigma,0.2\sigma)$ to break the exact grid symmetry.  Every
cluster contributes $\lfloor N/K\rfloor=8$ observations drawn from
$\mathcal{N}(\mu_k,\sigma^2 I_2)$, giving $N=200$.  The seed controls both the
center jitter and the observation draw.  Although the ratio $D/\sigma=5$ makes
the target clusters well separated, the small number of observations per
cluster makes $K$-means++ susceptible to placing two initial centers in one
Voronoi cell while leaving another cell empty. Table~\ref{tab:largeK} compares PIP against its own alternating warm start for
both $(p,q)=(2,1)$ and $(p,q)=(1,1)$.  

\begin{table}[htbp]
\begin{center}
\caption{Small-$N$/large-$K$ grid ($K=25$, $\sigma=1$, spacing
$D=5\sigma$, $N=200$, $\approx8$ observations/cluster).}
\label{tab:largeK}
\small
\setlength{\tabcolsep}{2.8pt}
\begin{tabular}{lrrrrrrrr}
\toprule
\multicolumn{9}{l}{\textit{Pair A:} \textsc{Alter} $(p{=}2,q{=}1)$ versus PIP $(p{=}2,q{=}1)$
  \hfill objective: $\mathrm{WCD}_2$} \\
\cmidrule(lr){1-9}
 & \multicolumn{3}{c}{$\mathrm{WCD}_2$} & \multicolumn{4}{c}{ARI} & \\
\cmidrule(lr){2-4}\cmidrule(lr){5-8}
Seed & \textsc{Alter} & PIP & $\%\Delta$ & \textsc{Alter} & PIP & $\Delta$ & $\%\Delta$ & Time (s) \\
\midrule
42   & 350.8 & 350.8 & $+0.0$  & 0.905 & 0.905 & $+0.000$ & $+0.0$  & 61.4  \\
7    & 398.1 & 328.6 & $-17.5$ & 0.857 & 0.944 & $+0.087$ & $+10.1$ & 78.6  \\
99   & 406.3 & 403.5 & $-0.7$  & 0.872 & 0.871 & $-0.001$ & $-0.1$  & 61.8  \\
123  & 345.5 & 323.4 & $-6.4$  & 0.900 & 0.956 & $+0.057$ & $+6.3$  & 62.9  \\
2024 & 443.7 & 436.1 & $-1.7$  & 0.816 & 0.820 & $+0.005$ & $+0.6$  & 195.9 \\
\textbf{avg}
     & $388.9{\pm}36.7$ & $368.5{\pm}44.1$ & $\mathbf{-5.2}$
     & 0.870 & 0.899 & $\mathbf{+0.029}$ & $\mathbf{+3.4}$
     & $92.1{\pm}52.3$ \\
\midrule\midrule
\multicolumn{9}{l}{\textit{Pair B:} $K$-medians versus PIP $(p{=}q{=}1)$
  \hfill objective: $\mathrm{WCD}_1$} \\
\cmidrule(lr){1-9}
 & \multicolumn{3}{c}{$\mathrm{WCD}_1$} & \multicolumn{4}{c}{ARI} & \\
\cmidrule(lr){2-4}\cmidrule(lr){5-8}
Seed & $K$-medians & PIP & $\%\Delta$ & $K$-medians & PIP & $\Delta$ & $\%\Delta$ & Time (s) \\
\midrule
42   & 290.2 & 288.4 & $-0.6$ & 0.866 & 0.884 & $+0.018$ & $+2.0$ & 422.1 \\
7    & 294.1 & 278.6 & $-5.3$ & 0.859 & 0.900 & $+0.041$ & $+4.7$ & 422.2 \\
99   & 308.4 & 289.9 & $-6.0$ & 0.877 & 0.926 & $+0.049$ & $+5.6$ & 362.1 \\
123  & 279.9 & 279.9 & $+0.0$ & 0.945 & 0.945 & $+0.000$ & $+0.0$ & 302.0 \\
2024 & 323.7 & 309.3 & $-4.5$ & 0.778 & 0.829 & $+0.051$ & $+6.6$ & 543.8 \\
\textbf{avg}
     & $299.2{\pm}15.3$ & $289.2{\pm}11.0$ & $\mathbf{-3.4}$
     & 0.865 & 0.897 & $\mathbf{+0.032}$ & $\mathbf{+3.7}$
     & $410.5{\pm}80.2$ \\
\bottomrule
\end{tabular}
\end{center}

\vspace{-0.1in}

{\small 
\noindent {\bf Note:}  We compare each PIP
method with its alternating warm start on the corresponding 
target objective.
$\%\Delta<0$ denotes an objective reduction and
$\Delta\mathrm{ARI}>0$ denotes improved recovery.}
\end{table}

PIP improves the target objective on four of five seeds for 
both norm pairs
and returns the warm-start value on the remaining seed.  For $(p,q)=(2,1)$,
the average $\mathrm{WCD}_2$ reduction is $5.2\%$ and the average ARI rises
from $0.870$ to $0.899$; seed $7$ exhibits the largest change, with a
$17.5\%$ objective reduction and an ARI increase of $0.087$.  For
$(p,q)=(1,1)$, the average $\mathrm{WCD}_1$ reduction is $3.4\%$ and average
ARI rises from $0.865$ to $0.897$.  Objective improvement does not force ARI
to improve on every instance, as illustrated by the change of $-0.001$ on
seed $99$ in Pair~A, but the average recovery improves for both pairs. The timing contrast is also substantial.  The $(2,1)$ PIP requires
$92.1\pm52.3$ seconds on average, compared with $410.5\pm80.2$ seconds for
the $(1,1)$ PIP under the same large-$K$ protocol.  The results therefore show
that PIP can improve naturally occurring $K$-means++ failures at $K=25$, but
also that computational cost depends strongly on the fitting loss.  A returned
warm start should be interpreted as ``no accepted improvement under the stated
search protocol'' unless the global-subproblem premise of the fixed-point
theorem has additionally been verified.

\subsection{PIP makes the MIP model \eqref{eq:MIP} 
computationally affordable} \label{sec:claim3}

The previous numerical results show that an IP approach can improve the quality of the solutions obtained by the alternating algorithm. In this section, we demonstrate that the PIP method is essential for making the IP approach tractable, even for modest-scale data. We demonstrate this argument on two
instances with $N=300$. The first instance is the 
truncated-Cauchy mixture of
Section~\ref{sec:adversarial-initialization} with $N=300$ observations, $75$
per cluster, at seed $2024$, initialized at the adverse centers
$\{(-1,7),(0,-4),(10,0),(14,0)\}$; it has $NK=1{,}200$ binary assignment
variables.  The second is a set of $K=8$ isotropic Gaussian blobs in
$\mathbb{R}^2$ with $N=300$ and $\mathcal{N}(\mu_k, I_2)$ noise, whose centers
are sampled at seed $42$ subject to a minimum pairwise separation of
$2.5\sigma$, and which is initialized by $K$-means++; it has $NK=2{,}400$
binary variables.  The two instances therefore probe the two ways the full
formulation can fail: a hard initialization at moderate $K$, and a larger
binary block at an easy initialization. The full MIP receives a $300$-second time budget,
whereas PIP applies the progressive restrictions developed in
Section~\ref{sec: PIP-q=1}. From the results shown in 
Table~\ref{tab:fullip}, the full IP cannot close its optimality gap
within a $300$\,s budget, in one case returning a solution
worse than
its own warm start, while PIP solves the identical instances 
to proven subproblem optimality in seconds. 

\begin{table}[h]
\begin{center}
\small
\caption{%
  Full-IP intractability demonstration
  (data size $N{=}300$). 
}
\label{tab:fullip}
\setlength{\tabcolsep}{6pt}
\begin{tabular}{lrrrrr}
\toprule
Method & Bin.\ vars & WCD & ARI & Time (s) & MIPGap \\
\midrule
\multicolumn{6}{l}{\textit{(a) 4-Cauchy mixture, $K{=}4$ ($1{,}200$ binary variables)}} \\
Alter L2/L1 (warm start) &    0 & 5531.6  & 0.692 &   0.0 & — \\
Full IP & 1200 & 10681.3 & 0.687 & 300.0 & 96.8\% \\
PIP ($p=2, q=1$) & $\leq NK$ & 365.6 & 1.000 & 0.6 & 0 \\
\midrule
\multicolumn{6}{l}{\textit{(b) Isotropic Gaussian blobs, $K{=}8$ ($2{,}400$ binary variables)}} \\
Alter L2/L1 (warm start) &    0 & 559.2 & 0.985 &   0.0 & — \\
Full IP  & 2400 & 559.2 & 0.985 & 300.0 & 100\% \\
PIP ($p=2, q=1$) & $\leq NK$ & 556.8 & 1.000 & 122.2 & 0 \\
\bottomrule
\end{tabular}
\end{center}

\vspace{-0.1in}

{\small {\bf Note:} Full IP is given a $300$\,s budget; 
PIP solves to proven optimality (MIPGap\,${=}\,0$).
Dataset (a): adversarial initial solution: 
$\{({-}1,7),(0,{-}4),(10,0),(14,0)\}$.
Dataset (b): KMeans$++$ is set as initial solution.}
\end{table}

\paragraph{Why Full IP is intractable.}
A direct mixed-integer formulation of $\ell_{p,q}$ clustering introduces $NK$
binary assignment variables $z_{sk}\in\{0,1\}$, one for every
(point, cluster) pair.
Even at the modest scale of Table~\ref{tab:fullip}, this yields $1{,}200$
variables for $K{=}4$ and $2{,}400$ for $K{=}8$.
The core difficulty is a \emph{weak LP relaxation}: the big-$M$ linearization
of the Heaviside assignment constraints allows the LP solver to set fractional
$z_{sk}$ in ways that drastically underestimate the true integer optimum,
producing lower bounds far below any achievable clustered solution.
As a result, the branch-and-bound tree is enormous and {\sc Gurobi} cannot close
the gap within 300\,s.

Two failure modes are evident.
In case~(a), Full IP returns a solution \emph{worse} than the alternating
warm start (WCD\,$10{,}681$ vs.\ $5{,}532$) with a gap of $96.8\%$, meaning
the B\&B search has barely explored the feasible region in 5 minutes.
In case~(b), the LP lower bound is so loose that {\sc Gurobi} reports an infinite
gap: even though a good warm-start incumbent (WCD\,$559.2$, ARI\,$0.985$) is
available, the solver cannot certify that any improving solution exists within
the budget.

PIP avoids this by decomposing the problem into a \emph{sequence of
progressive sub-IPs}.
At each $\tau$-level, only the ambiguous assignment rows (those with
$\rho_{sk}\in(1/K{-}\varepsilon,\,0.5)$) remain as free binary variables;
high-confidence assignments are fixed, shrinking the active B\&B tree to a
fraction of the full problem.
The smaller, tighter sub-IPs converge rapidly: PIP certifies optimality
(MIPGap\,${=}\,0$) and recovers ARI\,${=}\,1.000$ in $0.6$\,s on the Cauchy
instance and $122.2$\,s on the Gaussian instance — orders of magnitude faster
than Full IP at equivalent or superior solution quality.

\subsection{Behavior when the alternating solution is already strong}
\label{sec:quality-support}

PIP is most visibly useful when it escapes a poor alternating solution, but a
complete evaluation should also examine the ordinary case in which the warm
start is already strong.  Table~\ref{tab:quality-support} collects nine
standard settings spanning well-separated, unequal-variance, anisotropic,
low-rank, and higher-dimensional mixtures. The first two settings reuse the four-component Gaussian and truncated-Cauchy
mixtures of Section~\ref{sec:adversarial-initialization}, now with $K$-means++
initialization.  Four low-dimensional settings use $N=300$ observations in
$\mathbb{R}^2$: isotropic blobs ($K=6$, unit variance, centers separated by at
least $6$); unequal-variance blobs ($K=5$, cluster standard deviations linearly
spaced from $0.3$ to $4.0$); anisotropic blobs ($K=5$, isotropic blobs passed
through the fixed shear $T=\bigl[\begin{smallmatrix}3&1.5\\0.5&1\end{smallmatrix}\bigr]$);
and orthogonal low-rank blobs ($K=3$, needle-shaped clusters with aspect ratio
$10$ elongated along $0^{\circ}$, $90^{\circ}$, and $45^{\circ}$).  Three
higher-dimensional settings use $N=1000$, $K=10$, and $d=5$: isotropic,
anisotropic (a fixed random shear normalized to the Frobenius norm of the
two-dimensional shear above), and low-rank blobs with aspect ratio $15$.  All
nine settings use $K$-means++ initialization and the standard seed set. Each data set is run with five
seeds.  The final column reports how often PIP returns the heuristic objective
unchanged (C) and how often it improves it (I).

\begin{table}[htbp]
\begin{center}
\caption{PIP behavior under $K$-means++ initialization. }
\label{tab:quality-support}
\small
\setlength{\tabcolsep}{4pt}
\begin{tabular}{l l r r r r c}
\toprule
Data set & Pair & Alter.\ obj. & PIP obj. & PIP ARI & Time (s) & C/I \\
\midrule
4-Gaussian ($K{=}4$, $N{=}600$)
  & $(2,1)$ & 2662.9 & 2662.9 & 1.000 & 0.7 & 5/0 \\
  & $(1,1)$ & 1423.4 & 1423.4 & 1.000 & 0.4 & 5/0 \\
\midrule
4-Cauchy ($K{=}4$, $N{=}200$)
  & $(2,1)$ & 262.2 & 262.2 & 1.000 & 0.1 & 5/0 \\
  & $(1,1)$ & 168.9 & 168.9 & 1.000 & 0.1 & 5/0 \\
\midrule
Isotropic ($K{=}6$, $d{=}2$)
  & $(2,1)$ & 587.6 & 587.6 & 1.000 & 58.5 & 5/0 \\
  & $(1,1)$ & 471.2 & 471.2 & 1.000 & 24.9 & 5/0 \\
\midrule
Unequal-var ($K{=}5$, $d{=}2$)
  & $(2,1)$ & 3740.2 & \textbf{3734.3} & 0.998 & 1.3 & 3/2 \\
  & $(1,1)$ & 1017.0 & 1017.0 & 0.998 & 1.0 & 5/0 \\
\midrule
Anisotropic ($K{=}5$, $d{=}2$)
  & $(2,1)$ & 3378.0 & \textbf{3366.1} & 0.904 & 51.5 & 3/2 \\
  & $(1,1)$ & 1103.3 & \textbf{1103.1} & 0.894 & 66.9 & 4/1 \\
\midrule
Orth.\ low-rank ($K{=}3$, $d{=}2$)
  & $(2,1)$ & 29617.8 & 29617.8 & 0.998 & 0.3 & 5/0 \\
  & $(1,1)$ & 2890.3 & 2890.3 & 0.998 & 0.1 & 5/0 \\
\midrule
Isotropic ($K{=}10$, $d{=}5$)
  & $(2,1)$ & 4958.8 & 4958.8 & 1.000 & 90.4 & 5/0 \\
  & $(1,1)$ & 3963.4 & 3963.4 & 1.000 & 117.5 & 5/0 \\
\midrule
Anisotropic ($K{=}10$, $d{=}5$)
  & $(2,1)$ & 12258.5 & 12258.5 & 0.999 & 77.5 & 5/0 \\
  & $(1,1)$ & 6028.2 & 6028.2 & 1.000 & 160.5 & 5/0 \\
\midrule
Low-rank ($K{=}10$, $d{=}5$)
  & $(2,1)$ & 228436.9 & 228436.9 & 1.000 & 8.7 & 5/0 \\
  & $(1,1)$ & 23351.8 & 23351.8 & 1.000 & 203.1 & 5/0 \\
\bottomrule
\end{tabular}
\end{center}

\vspace{-0.1in}

{\small \noindent {\bf Note:} Five seeds per data
set.  Every row compares one alternating--PIP pair on its own objective:
$\mathrm{WCD}_2$ for $(p,q)=(2,1)$ and $\mathrm{WCD}_1$ for
$(p,q)=(1,1)$.  All values in the columns 3-6 are the computed mean over seeds.  C/I is the number of seeds on
which PIP returned the alternating objective unchanged / improved it.}
\end{table}

Across the $45$ runs for each pair, PIP returns the alternating objective
unchanged in $41$ cases for $(p,q)=(2,1)$ and in $44$ cases for
$(p,q)=(1,1)$; it improves the remaining runs and never returns a worse target
objective.  These outcomes complement the earlier stress tests.  When
$K$-means++ already produces a high-quality solution, the progressive search
usually preserves it and incurs a cost ranging from less than one second to a
few minutes.  On the unequal-variance and anisotropic instances, it detects
small improvements missed by the alternating updates.

An unchanged outcome provides empirical evidence that the warm 
start is stable under the executed progressive search.  It 
becomes the local-optimality certificate established in 
Section~\ref{sec:theory} only when the relevant
self-generated subproblem is solved globally and all hypotheses 
of the corresponding theorem are satisfied.  This distinction 
separates the algorithm's unconditional feasible-descent property
from its conditional theoretical certificate.

% \section{Conclusion}
% \label{sec:conclusion}

\gap

\noindent {\bf Conclusion.}
This paper introduced $\ell_{p,q}$ clustering as a unified 
framework that separates the geometry used for centroid updates 
and the cluster assignments.  In our notation, the $\ell_q$ norm 
governs nearest-center assignment
and the $\ell_p$ norm governs within-cluster dispersion, which are allowed to be selected independently. This generalizes the classical K-means $(p = q = 2)$ and K-median models $(p = q = 1)$. To avoid exposing a general-purpose solver to solve all
$NK$ binary variables (for cluster assignments), we developed a progressive integer
programming method that adaptively fixes assignments based on the current centers and
retains integer variables only for ambiguous observation--cluster pairs.  For
$q=1$, difference-of-convex norm comparisons yield a convex mixed-integer inner
approximation, and the working set mechanism inserts the fixed-assignment
conditions that become necessary as the centers move.  The resulting
subproblems are MILPs for $p=1$ and convex MIQCPs for $p=2$ of much smaller dimensions. We give a precise characterization of the solutions obtained by the PIP method. Particularly, the global
optimality of the adaptively generated restricted subproblem implies strong
center-local optimality, and hence ordinary local optimality, for the original
mixed-norm model.  A converse holds under a practical sample-wise uniqueness assumption.  In the matched norm case $p=q$, we establish an equivalence between the global optimality of the simplified restricted
problem and strong center-local optimality, without requiring the sample-wise unique
centers assumption. For $q=1$, we further develop effective techniques for constructing adaptive fixing sets
and working sets. The numerical results illustrate distinct benefits of the mixed-norm clustering model and the PIP algorithm.
The preferred norm pair depends on the contamination geometry:
$\ell_{2,1}$ performs best under coordinate-sparse, mean-balanced corruption,
whereas $\ell_{1,2}$ is favored under dense coordinatewise Cauchy corruption.
Thus the mixed formulation provides flexibility that neither $K$-means nor
$K$-medians alone can offer.  Separately, we show that
alternating procedures can settle at poor feasible solutions, while PIP can escape
such solutions effectively. When the alternating
warm start is already strong, PIP usually preserves it and never worsens its
target objective in the reported experiments. The present results provide a
rigorous foundation for the mixed-norm clustering model and demonstrate that it can be both useful in noisy data clustering and computationally accessible.

\gap

\noindent {\bf Some questions remain open.}  
The present computational development
concentrates on $q=1$ and $p\in\{1,2\}$. A scalable treatment of the cases where $q = 2$ is an important next step.  Moreover, it would be valuable to weaken the separation assumptions required by
the convex inner approximation and to develop efficiently verifiable
certificates when the subproblems are solved inexactly.

\section*{Declarations}

\subsection*{Funding}

The work of Junyi Liu was supported by the National Natural Science
Foundation of China under Grant 72571155. The work of Yao Xie was
partially supported by the National Science Foundation under Grant
CMMI-2112533 and the Coca-Cola Foundation. The work of 
Yancheng Yuan was supported by the NSFC Young Scientists Fund 
(Project No.\ 12501440) and the RGC Early Career Scheme (Project
No.\ 25305424). The work of Yulin Peng and Jong-Shi Pang was 
supported by the U.S.\ Air Force
Office of Scientific Research under Grant FA9550-22-1-0045.

\subsection*{Conflict of interest}

The authors declare that they have no conflict of interest.

\subsection*{Code availability}

The code used for the numerical experiments is available from the
corresponding author upon reasonable request.


\begin{thebibliography}{999}

{\small

\bibitem{aloise2012improved}
{\sc D.\ Aloise, P.\ Hansen, and Leo Liberti}.
An improved column generation algorithm for minimum sum-of-squares 
clustering.
{\sl Mathematical Programming} 131: 195--220 (2012).

\vspace{-0.1in}

\bibitem{arthur2007kmeans++}
{\sc D.\ Arthur and S.\ Vassilvitskii}.
$k$-means{++}: The advantages of careful seeding.
{\sl Proceedings of the Eighteenth Annual ACM-SIAM Symposium on 
Discrete Algorithms (SODA)} (2007) pp.\ 1027--1035.

\vspace{-0.1in}

\bibitem{banerjee2005clustering}
{\sc A.\ Banerjee, S.\ Merugu, I.S.\ Dhillon, and J.\ Ghosh}.
Clustering with {Bregman} divergences.
{\sl Journal of Machine Learning Research} 6: 1705--1749 (2005).

\vspace{-0.1in}

\bibitem{bradley1997clustering}
{\sc P.S.\ Bradley, O.L.\ Mangasarian, and W.N.\ Street}.  
Clustering via concave minimization.
{\sl Advances in Neural Information Processing Systems} 
9: 368--374 (1997).

\vspace{-0.1in}

\bibitem{bradley1999mathematical}
{\sc P.S.\ Bradley, U.M.\ Fayyad, and O.L.\ Mangasarian}.
Mathematical programming for data mining: Formulations 
and challenges.
{\sl INFORMS Journal on Computing} 11(3): 217--238 (1999).

\vspace{-0.1in}

\bibitem{bradley2000kplane}
{\sc P.S.\ Bradley, O.L.\ Mangasarian}.
$k$-Plane clustering.  {\sl Journal of Global Optimization} 16(1): 
23--32 (2000).

\vspace{-0.1in}

\bibitem{cohen2021inapproximability}
{\sc V.\ Cohen-Addad, C.S.\ Karthik, and E.\ Lee}.
On approximability of clustering problems without candidate centers.
{\sl Proceedings of the 2021 ACM-SIAM Symposium on Discrete 
Algorithms (SODA)} (2021) pp.\ 2635--2648.

\vspace{-0.1in}

\bibitem{CuiLiuPang23-piecewise}
{\sc Y.\ Cui, J.\ Liu, and J.S.\ Pang}.
The minimization of piecewise functions: Pseudo stationarity.
{\sl Journal of Convex Analysis} 30(3): 793--834 (2023).

\vspace{-0.1in}

\bibitem{FangLiuPang25}
{\sc Y.\ Fang, J.\ Liu, and J.S. Pang}. 
Treatment learning with Gini constraints by Heaviside composite
optimization and a progressive method. 
{\sl Computational Optimization and Applications} 
92: 471--513 (2025).

\vspace{-0.1in}

\bibitem{gurobi}
{\sc gurobi Optimization, LLC}.
GUROBI Optimizer Reference Manual. (2023).
\url{https://www.gurobi.com}.

% \bibitem{gribel2019hg}
% {\sc D.\ Gribel and T.\ Vidal}.
% {HG}-means: {A} scalable hybrid genetic algorithm for minimum 
% sum-of-squares clustering.
% {\sl Pattern Recognition} 88: 569--583 (2019).

% \bibitem{gu2011joint}
% {\sc Q.\ Gu, Z.\ Li, and J.\ Han}.
% Joint feature selection and subspace learning.
% {\sl Proceedings of the Twenty-Second International 
% Joint Conference
% on Artificial Intelligence ({IJCAI})} (2011) pp.\ 1294--1299.

\vspace{-0.1in}

\bibitem{HanCuiPang25}
{\sc S.\ Han, Y.\ Cui, and J.S.\ Pang}.
Analysis of a class of minimization problems lacking 
lower semicontinuity.
{\sl Mathematics of Operations Research} 50(3): 2175--2198 (2025).

\vspace{-0.1in}

\bibitem{herold2026clustering}
{\sc M.G.\ Herold, E.\ Kipouridis, and J.\ Spoerhase}.
A broader view on clustering under cluster-aware norm objectives.
{\sl Proceedings of the 2026 Annual {ACM-SIAM} Symposium on 
Discrete Algorithms ({SODA})} (2026) pp.\ 758--793.

% \bibitem{jain2010clustering}
% {\sc A.K.\ Jain}.
% Data clustering: 50 years beyond {K}-means.
% {\sl Pattern Recognition Letters} 31(8): 651--666 (2010).

\vspace{-0.1in}

\bibitem{kaufman2009}
{\sc L.\ Kaufman and P.J.\ Rousseeuw}.
{\sl Finding Groups in Data: An Introduction to Cluster Analysis}.
John Wiley \& Sons (2009).

\vspace{-0.1in}

\bibitem{li2025modified}
{\sc M\ Li, M.R.\  Metel, and A.\ Takeda}.
Modified {K}-means algorithm with local optimality guarantees.
{\sl Proceedings of the 42nd International Conference on 
Machine Learning} (2025).
\url{https://openreview.net/forum?id=v68wPjgEQ8}

\vspace{-0.1in}

\bibitem{Lloyd1982}
{\sc S. Lloyd}.
Least squares quantization in PCM.
{\sl IEEE Transactions on Information Theory} 28:129--137 (1982).

\vspace{-0.1in}

% \bibitem{nie2010efficient}
% {\sc F.\ Nie, W.\ Huang, X.\ Cai, and C.h.\ Ding}.
% Efficient and robust feature selection via joint 
% $\ell_{2,1}$-norms 
% minimization.
% {\sl Advances in Neural Information Processing Systems} 
% 23: 1813--1821
% (2010).

\bibitem{macqueen1967}
{\sc J.\ MacQueen}.
Some methods for classification and analysis of 
multivariate observations.
Editors: L.M.\ Le Cam and J.\ Neyman.
{\sl Proceedings of the Fifth Berkeley Symposium on Mathematical
Statistics and Probability} (1967) pp.\ 281--297.

\vspace{-0.1in}
\bibitem{mahajan2009planar}
{\sc M.\ Mahajan, P.\ Nimbhorhar, and K.\ Varadarajan}.
The planar $k$-means problem is {NP}-hard.
{\sl Theoretical Computer Science} 442: 13--21 (2012).

\vspace{-0.1in}

\bibitem{PangRazaAlvarado17}
{\sc J.S.\ Pang, M.\ Razaviyayn, and A.\ Alvarado}.
Computing B-stationary points of nonsmooth dc 
programs.  {\sl Mathematics of Operations Research} 
42: 95--118 (2017).

\vspace{-0.1in}

\bibitem{peng2007approximating}
{\sc J.\ Peng and Y.\ Wei}.
Approximating {K}-means-type clustering via semidefinite 
programming.
{\sl SIAM Journal on Optimization} 18(1): 186--205 (2007).

\vspace{-0.1in}

\bibitem{piccialli2022sossdp}
{\sc V.\ Piccialli, A.M.\ Sudoso, and A.\ Wiegele}.
{SOS-SDP}: An exact solver for minimum sum-of-squares clustering.
{\sl INFORMS Journal on Computing} 34(4): 2144--2162 (2022).

\vspace{-0.1in}

\bibitem{QiCuiLiuPang19}
{\sc Z.\ Qi, Y.\ Cui, Y.\ Liu, and J.S.\ Pang}.
Estimation of individualized decision rules based on optimized 
covariate-dependent equivalent of random outcomes.
{\sl SIAM Journal on Optimization} 29(3): 2337--2362 (2019).

\vspace{-0.1in}

\bibitem{rao1971cluster}
{\sc M.R.\ Rao}.
Cluster analysis and mathematical programming.
{\sl Journal of the American Statistical Association} 66(335): 
622--626 (1971).

\vspace{-0.1in}

\bibitem{selim1984kmeans}
{\sc S.Z.\ Selim, and M.A.\ Ismail}.
{K}-means-type algorithms: A generalized convergence theorem 
and characterization of local optimality.
{\sl IEEE Transactions on Pattern Analysis and Machine Intelligence}
6(1): 81--87 (1984).

\vspace{-0.1in}

\bibitem{VielmaAhmedNemhauser10}
{\sc J.P.\ Vielma, S.\ Ahmed, and G.\ Nemhauser}.
Mixed-integer models for nonseparable piecewise-linear 
optimization: unifying framework and extensions.
{\sl Operations Research} 58(2): 303--315 (2010).

\vspace{-0.1in}

\bibitem{ZhangHanPang26}
{\sc X.\ Zhang, S.\ Han, and J.S.\ Pang}. 
Improving the solution of indefinite quadratic programs and linear 
programs with complementarity constraints by a 
progressive MIP method.
{\sl Mathematical Programming Computation} 18(1): 135--182 
(2026-03).

\vspace{-0.1in}

\bibitem{ZhengLiuWangPang26}
{\sc K.\ Zheng, J.\ Liu, R.\ Wang, and J.S.\ Pang}. 
Solving constrained affine Heaviside composite
optimization problems by a progressive IP approach. 
\url{https://arxiv.org/abs/2605.06020} (April 2026).

}
	
\end{thebibliography}
\end{document}